\documentclass{article}

\title{Stochastic Mackey--Glass Equations and Other  Negative Feedback Systems: Existence of Invariant Measures} 
\author{M. van den Bosch$^{\rm a, }$\footnote{Corresponding author.\\    Email addresses: \url{mark-bosch@hotmail.com}, \url{vangaans@math.leidenuniv.nl} and \url{S.M.VerduynLunel@uu.nl}.\\ To appear in SIAM
Journal on Applied Dynamical Systems.}\,\,, O.\,W. van Gaans$^{\mathrm a}$, S.\,M. Verduyn Lunel$^{\mathrm b}$}

\date{\today}

\usepackage{todonotes}
\usepackage{lipsum} 

\usepackage{subfiles}
\usepackage{xr}

\usepackage{geometry}
\usepackage[utf8]{inputenc}
\usepackage[english]{babel}
\usepackage{graphicx} 
\usepackage{xcolor}

\usepackage{mathtools} 
\usepackage{amsmath} 
\usepackage{amssymb} 
\usepackage{amsthm}
\usepackage{xspace}
\usepackage{enumerate}
\usepackage{mathrsfs}  

\usepackage[hidelinks]{hyperref}
\usepackage{biblio}
\numberwithin{equation}{section}
\newcommand\yesnumber{\addtocounter{equation}{1}\tag{\theequation}}

\newtheorem{theorem}{Theorem}[section]
\newtheorem{lemma}[theorem]{Lemma}
\newtheorem{proposition}[theorem]{Proposition}
\newtheorem{corollary}[theorem]{Corollary}

\theoremstyle{definition}
\newtheorem{definition}[theorem]{Definition}
\newtheorem{remark}[theorem]{Remark}
\newtheorem{assumption}[theorem]{Assumption}

\newtheorem{examplex}[theorem]{Example}
\newenvironment{example}
{\pushQED{\qed}\examplex}
{\popQED\endexamplex}

\newcommand{\cadlag}{càdlàg\xspace}
\newcommand{\caglad}{càglàd\xspace}
\newcommand{\R}{\mathbb{R}}
\newcommand{\Rplus}{[0,\infty)}
\newcommand{\tauRplus}{\mathbb D[-\tau,\infty)}
\newcommand{\taunul}{\mathbb D[-\tau,0]}
\newcommand{\N}{\mathbb{N}}
\newcommand{\dint}[3]{\int_{#1}^{#2}{#3} \, \mathrm{d}}

\def\gb #1{\bigl( #1 \bigr)}
\def\ph {\varphi}

\def\ph
{\varphi}

\def\th
{\theta}

\usepackage{tcolorbox}

\tcbuselibrary{breakable}
\tcbset{
  width=0.6\textwidth,
  halign=justify,
  center,
  breakable,
  colback=white, 
  colframe=black!50
}

\begin{document}

\maketitle
\begin{center}\small
    \textsc{
    $^{\mathrm a}$Mathematical Institute,  Leiden University,\\ P.O. Box 9512, 2300 RA Leiden, The Netherlands}

    \
    
    \textsc{$^{\mathrm b}$Department of Mathematics, University of Utrecht,\\ P.O. Box 80010, 3508 TA Utrecht, The Netherlands}
\end{center}

\

\begin{abstract}
   \noindent We study
   equations like the Mackey--Glass equations and Nicholson's blowflies equation, each perturbed by a (small) multiplicative noise term. Solutions to these stochastic negative feedback systems persist globally  and are bounded above in probability under mild assumptions. 
   A non-trivial invariant measure is proved to exist if and only if  there is at least one  initial condition  for which the solution remains bounded away from zero in probability. The noise driving the dynamical system is allowed to be a square integrable Lévy process with finite intensity. Existence of   invariant measures is obtained via the Krylov--Bogoliubov method.   In addition to our theoretical results, we present numerical simulations identifying the invariant measures obtained via the Krylov–Bogoliubov method and illustrating their connection to the system’s long-term behaviour.
\end{abstract}


\section{Introduction}

In this paper, we investigate the existence of invariant (probability) measures for stochastic delay differential equations whose deterministic counterpart is of the form 
\begin{equation}
	 x'(t)=-\gamma (t)x(t)+r(t)f(x(t-\tau)),\label{eq:delay-1}
\end{equation}
where $\tau>0$ characterises a fixed time delay. Viewing \eqref{eq:delay-1} as a population dynamics model,  $\gamma(t)$ and $r(t)$ can be interpreted as  time-dependent  mortality and reproduction rates, respectively; they are assumed to be measurable, positive for all  $t\geq 0$, and are usually   constant. We take  $f:\R\to\R$  to be a non-negative continuous function.  
 By a  time rescaling argument, one may set $\tau=1$ without loss of generality.


A wide variety of
 systems  exist in which   (quasi-)periodic    behaviour   of certain components
 have been  observed \cite{kolmanovskii2013introduction}. A  possible mechanism causing the latter is negative feedback  \cite{glass1990chaos,milton1989complex}. 
 Negative feedback occurs whenever some function of the  output of a system is  fed back into the system in such a way that it tends to reduce the increase or fluctuations of the output.
Numereous homeostatic processes in physiological systems rely on negative feedback in order to control the concentration of substances in   blood. Breathing, for  example, is stimulated   whenever the brain detects|with a certain delay|a high carbon dioxide concentration  in    blood \cite{berezansky2012mackey,glass1979pathological}. Throughout this work, the time delay is assumed to be  only within the reproduction term. Our usage of the terminology  ``negative feedback'' should not be confused with that of  \cite{book:hale,krisztin2008global,rost2007domain}. Following the terminology there, the Mackey--Glass equation|with $q=1$|and the Nicholson's blowflies equation are examples of unimodal feedback.


The well-known Mackey--Glass equations are given by system \eqref{eq:delay-1} after setting the nonlinearity  $f:[0,\infty)\to\mathbb R$ to\footnote{It is not difficult to see that \eqref{eq:Mackey-Glass} can be smoothly extended to the negative reals $(-\infty,0)$ in such a manner that   $f$ becomes a globally Lipschitz functional on $\mathbb R.$} 
\begin{equation}
	 f(x)=\frac{x^q}{1+x^p},\qquad p\geq 1,\quad q\in\{0,1\}.\label{eq:Mackey-Glass}
\end{equation} 
 These equations were initially  proposed to model the concentration of white blood cells \cite{glass1979pathological,article:mackeyglass}. The delay term was introduced to account for the significant time lag between detecting low white blood cell concentration and the bone marrow producing and releasing mature cells into the blood. Numerical observations showed relevant cycling behaviour for 
$q=1$, providing heuristic justification for the model. While unimodal feedback
($q=1$) can lead to chaotic dynamics, it is not possible with monotone feedback ($q=0$) \cite{rost2007domain}.

The class \(\eqref{eq:delay-1}\) of delay differential equations arises as a mathematical model for various other fields as well, e.g., population dynamics. The Nicholson's blowflies equation, given by
\begin{equation}
    x'(t) = -\gamma(t)x(t) + r(t)x(t-1)e^{-px(t-1)}, \qquad p>0,\label{eq:Nicholson}
\end{equation}
 was introduced to model isolated laboratory insect populations \cite{gurney1980nicholson} and  $x(t)$ denotes the population of sexually mature adults at time $t$. The associated functional $f(x) = xe^{-px}$ satisfies $f(x)\to 0$ as $x\to\infty$, which corresponds to intraspecific competition. Nicholson's experiments on Australian sheep blowflies showed cycles with periods close to the generation time \cite{nicholson1954outline}. This (quasi-)periodic behaviour is attributed to a delay of $\tau$ $(\tau = 1)$ time units, representing the time needed for an egg to hatch and  for the   organism to   develop into a sexually mature adult. Equation \eqref{eq:Nicholson} replicates these  cycles, as demonstrated in \cite{gurney1980nicholson}.

\paragraph{Key objectives} 
Chaos  has  not yet been rigorously proved for the Mackey–Glass equation or Nicholson’s blowflies equation \cite{henot2023numerical,krisztin2008global,walther2020impact}, though it has been demonstrated for other non-monotone feedback systems; see \cite{lani1996chaotic}. 
 The complexity  of these equations have attracted significant attention lately.  
Recently,  existence of global attractors,  including some of its properties, has  been established  \cite{berezansky2013mackey,krisztin2020periodic,liz2008global,LizRost10,rost2007domain,wei2007bifurcation}, but these only concern  steady states and periodic orbits. In the case of monotone feedback systems like Wright's equation,  results on global attractors are already known for a while now \cite{book:hale,walther19952}. Even more recently, the article \cite{henot2023numerical} establishes the existence of a transverse homoclinic orbit, and hence a hyperbolic horseshoe with an abundance of chaotic trajectories. While this does not yield a strange attractor, whose presence would imply (physical) invariant measures \cite{eckmann1985ergodic}  and sustained long-time chaos, it does represent the current state of the art. Over the past fifty years, many theoretical and numerical advancements have been achieved in this area; see \cite{bartha2021stable,article:MGanalysis,berezansky2010nicholson,duruisseaux2022bistability,farmer1982chaotic,book:hale,henot2023numerical,junges2012intricate,kiss2017controlling,kojima2025resonance,krisztin2008global,lasota1977ergodic,liz2008global,LizRost10,lossom2020density,mitkowski2012ergodic,rost2007domain,walther2020impact} and the  references therein. 

In this paper, we contribute to the literature by following a probabilistic approach to explore the rich dynamics of equation \eqref{eq:delay-1}, as suggested in the monograph \cite{lossom2020density}. We aim to uncover ``hidden" structures in \eqref{eq:delay-1} by means of proving the existence of a \textit{non-trivial} invariant measure; this measure is distinct from the  Dirac measure 
$\delta_0$,
  associated to the fixed point 
$x=0$ in case $f(0)=0$. While we cannot  confirm the presence of a strange attractor with a mathematical proof so far, our analytical findings (in combination with numerics) offer a partial resolution to a major open problem by relating order to chaos via ergodicity. 

Robustness of these structures  is illustrated by stochastically perturbing  \eqref{eq:delay-1}
 and searching for non-trivial invariant measures within these stochastic negative feedback systems, as was done with Wright's equation \cite{BosGaaVer1-24}. The former work primarily focuses on the necessary estimates, highlighting the novel approach where one needs to understand the reverse time supremum of It\^o-driven processes with negative drift (see also {\S}\ref{Sec5}), and cites the results in this paper regarding invariant measures, tightness, boundedness in probability,  and (local) existence and uniqueness of solutions. In here, we aim to be comprehensive by developing general tools and results useful to others in the field.




\paragraph{Noisy systems} 

Let us discuss the class of stochastic perturbations of 
\begin{equation}
	 x'(t)=-\gamma (t)x(t)+r(t)f(x(t-\tau)),\quad x_0=\ph,\,\ph\in D[-\tau,0],\label{eq:delay-2}
\end{equation}
 which we  analyse in this paper. Our results  do not restrict to Brownian noise only;  see {\S}\ref{sec:typesnoise} for the noise processes we allow.  We write  $D\Rplus$, $D[-\tau,\infty),$ and $D[-\tau,0]$ for the spaces of \cadlag functions defined on $[0,\infty)$, $[-\tau,\infty)$, and $[-\tau,0]$,  respectively, along with $C\Rplus$, $C[-\tau,\infty),$ and $C[-\tau,0]$  being the  spaces of  continuous functions. We denote by $x_t=(x(s))_{s\in[t-\tau,t]}$ the segment of the solution at time $t.$  See \cite[App. B]{artikel1-general} for properties of  the segment spaces $C[-\tau,0]$ and $D[-\tau,0]$.
 
Under  mild assumptions, one can show in line with  \cite{article:MGanalysis}  that  initial value problem \eqref{eq:delay-2} satisfies the permanence property\footnote{When $f(0)>0$ holds, as for the monotone Mackey--Glass equations $(q=0)$, one can drop the $\ph(0)>0$ condition.}: for any non-negative and non-zero initial function $\ph$ with $\ph(0)>0$ the corresponding solution is positive for all $t\geq 0$.  In order to find a class of stochastic perturbations that preserve this positivity property, we use the same idea as for Wright's equation in \cite{BosGaaVer1-24}. We consider the transformation 
\begin{equation}y(t) = \log x(t),\quad x(t) = e^{y(t)},\label{eq:transformation}\end{equation}
and the transformed delay differential equation becomes
\begin{equation}\label{eq:t-mono-MG2}
y'(t) = - \gamma(t) + r(t)e^{-y(t)}f(e^{y(t-\tau)}).
\end{equation}

Throughout this paper, we analyse the equation \eqref{eq:stoch-t-mono-MG2}. That is, noise is added to the transformed system \eqref{eq:t-mono-MG2}  in the It\^o-sense \cite{book:karatzas,book:mao,book:revuz} and we include an extra drift term:
\begin{equation}\label{eq:stoch-t-mono-MG2}
\mathrm dY(t) = \bigl[- \gamma(t) + r(t)e^{-Y(t)}f(e^{Y(t-\tau)})]\mathrm dt + \underbrace{a(Y_t,t)\,\mathrm dt}_{\substack{\text{artificial}\\[.1cm]\text{drift\,term}}} + \underbrace{b(Y_t,t) \,\mathrm dW(t)}_{\text{noise term}},
\end{equation}
where $W$ is a standard Brownian motion on a filtered  probability space $(\Omega,\mathcal F,\mathbb F,\mathbb P)$,   $\mathbb F=(\mathcal F_t)_{t\geq 0},$ satisfying the usual conditions and $a,b : C[-\tau,0]\times \mathbb R \to \mathbb R$ are locally Lipschitz  with respect to the supremum norm $\|\cdot\|_\infty$ in the first coordinate. In case one replaces $W$ by, e.g., a  L\'evy process, we   substitute $C[-\tau,0]$ by $D[-\tau,0]$ and a few  additional assumptions on $a$ and $b$ are required; see \cite[Sec. 2]{artikel1-general}. In this more general setting, the rates $\gamma(t)$ and $r(t)$ need to be \cadlag  to ensure well-posedness. 

To see the effect of the noise term in \eqref{eq:stoch-t-mono-MG2} in terms of the original variables, we   use It\^o's formula  with $X(t)=\Psi(Y(t))$ and $\Psi(y) = e^y$ to conclude that
\begin{equation}
\begin{aligned}
\mathrm dX(t)&= \Psi'(Y(t))\,\mathrm dY(t) + \tfrac{1}{2}\Psi''(Y(t))b(Y_t,t)^2\,\mathrm dt\\
 & =\big[-\gamma(t)X(t)+r(t)f(X(t-\tau))\big]\mathrm dt\\&\qquad\quad+X(t)\big[\tilde a(X_t,t)+\tfrac12\tilde b(X_t,t)^2\big]\,\mathrm dt+X(t)\tilde b(X_t,t)\,\mathrm dW(t),\label{eq:stoch-x}
\end{aligned}
\end{equation}
where for $\ph \in C[-1,0]$ with $\ph > 0$ we have  defined
\begin{equation}
\tilde a(\ph) = a\gb{\th \mapsto \log \ph(\th)}\quad\hbox{and}\quad\tilde b(u) = b\gb{\th \mapsto \log \ph(\th)}.
\end{equation}
Therefore, in order to merely add a multiplicative noise term to the original equation \eqref{eq:delay-2}, we need to impose the following condition on the coefficients $a$ and $b$ in  \eqref{eq:stoch-t-mono-MG2}:
\begin{equation}\label{eq:neg-drift}
a(\ph,t) = -\frac{1}{2}b(\ph,t)^2,\quad \ph\in C[-\tau,0].
\end{equation}
In other words, adding noise to  \eqref{eq:delay-2} results into a negative drift term in the transformed equation \eqref{eq:stoch-t-mono-MG2}. Further, note that if we start with  \eqref{eq:stoch-x} and assume that the functionals $\tilde a$ and $\tilde b$  are locally Lipschitz in the first coordinate, then  this holds true for $a$ and $b$ in \eqref{eq:stoch-t-mono-MG2} as well. 
Ultimately, if  we replace $W$ by a L\'evy process $L$ of finite intensity, i.e., $\nu(\mathbb R)<\infty$ where $\nu$ denotes the L\'evy measure of $L$, then imposing the condition
\begin{equation}
    a(\ph,t)=-\frac{b(\ph,t)^2}{2}+b(\ph,t)\int_{|z|\leq 1}z\nu (\mathrm dz),\quad \ph\in D[-\tau,0],
\end{equation}
 gives rise to a L\'evy stochastic differential equation \cite{book:applebaum,article:jacob,book:sulem} of the form
\begin{equation}
	\begin{aligned}
	\mathrm d X(t) 
	&=\left[-\gamma(t) X(t)+r(t)f(X(t-\tau))\right] \mathrm d  t+ \sigma X(t)\tilde b(X_{t},t ) \,\mathrm  d W(t) \\& 
	 \hspace{4cm}+\int_{\mathbb{R}}X(t-)\big[\exp \big(\tilde b (X_{t-},t- )\,z\big)-1\big]  {\mu_L}(\mathrm d t, \mathrm d z), \label{eq:levysde-intro}
\end{aligned} 
\end{equation} 
 where $\mu_L$ is the random jump measure \cite{book:billingsley,book:jacod,book:kyp,book:sato} associated to the L\'evy process $L$. Observe that applying an It\^o's formula as in \cite{book:protter} does not directly result into a genuine differential equation.

\paragraph{Main results}  We will mainly analyse the transformed stochastic equation \eqref{eq:stoch-t-mono-MG2} throughout this paper, as it allows us to find and search for
\textit{non-trivial} invariant measures\footnote{We first show the existence of an invariant measure $\nu$ on $D[-\tau,0]$, implying the existence of a stationary distribution $\mu$ which is an invariant measure on $\mathbb R.$ In this paper, we keep a clear distinction between the two terminologies;  see {\S}\ref{sec:typesnoise}}. In this work, note that  non-trivial means that it is distinct from the trivial measure $\delta_0$, associated to the fixed point $x=0$ when $f(0)=0$. Existence of invariant measures  is achieved via the Krylov--Bogoliubov procedure. For this  we use results on generic stochastic delay differential equations from our other article \cite{artikel1-general}. Looking at the transformed equation \eqref{eq:stoch-t-mono-MG2} prevents us from constructing $\delta_0.$

Limiting ourselves to   Brownian noise  yields the following theorem. Similar results hold for certain types of L\'evy processes; see {\S}\ref{Sec6}. 


\begin{theorem}[see {\S}\ref{Sec6}]\label{main-thrm}
    Suppose $f:\mathbb R\to \mathbb R$ is  locally Lipschitz continuous, non-negative on $(0,\infty)$, and bounded from above. Assume $\inf_{t\geq 0}\gamma(t)>0$, $\sup_{t\geq 0}r(t)<\infty$ and  $\sigma\geq 0$. 
    Then the solution to
    \begin{equation}
        \mathrm dX(t)=[-\gamma(t) X(t)+r(t)f(X(t-\tau))]\,\mathrm dt+ \sigma X(t)c(X_t)\,\mathrm dW(t),\quad X_0=\Phi,\label{eq:delay-noise}
    \end{equation}
  where $c:C[-\tau,0]\to\mathbb R$ is  bounded and  locally Lipschitz  with respect to the supremum norm $\|\cdot\|_\infty$, is
unique,   persists globally, remains positive, and is bounded in probability for almost every  positive $\mathcal F_0$-measurable random variable $\Phi$ taking values in $C[-\tau,0]$.
 Furthermore, when $\gamma$ and $r$ are constants and if
    \begin{itemize}
        \item  $f(0)>0,$ then almost every solution is bounded away from zero in probability and there  exists a stationary distribution $\mu$ with $\mu((0,\infty))=1;$ 
        \item  $f(0)=0,$ then there exists at least two distinct stationary distributions, namely the Dirac measure $\delta_0$ and a   stationary distribution $\mu$ with $\mu((0,\infty))=1,$ provided that there is at least one initial value $\Phi$ such that  the solution is bounded away from zero in probability.
    \end{itemize}
    In fact, the distribution $\mu$   is the push-forward measure of an invariant measure $\nu$ on $C[-\tau,0]$ under the evaluation map $C[-\tau,0]\to\mathbb R, \varphi\mapsto \varphi(0).$
   
\end{theorem}

For the first part of Theorem \ref{main-thrm}, one may also consider $c$ that explicitly depends on time. When $
c$ is constant,  one derives equation \eqref{eq:delay-noise}  from \eqref{eq:delay-1} by  stochastically perturbing the mortality parameter in the latter according to the formal substitution
$\gamma(t)\mapsto \gamma(t)+\sigma \dot W(t)$. While this may be a useful practice for handling measurement errors, it  also   reflects the fact that noise is an intrinsic property of, e.g.,  physiological systems \cite{glass2001synchronization}. Additionally,  not always there is a clear distinction between noisy and chaotic behaviour; the source of the multiplicative noise may itself trace a chaotic process \cite{milton1989complex}.

  Theorem \ref{main-thrm} concludes our investigation with regard to the existence of a non-trivial invariant measure  for the stochastic monotone Mackey--Glass equation and complements the main findings in \cite{BosGaaVer1-24}, which studies the stochastic Wright's equation.  For the deterministic Mackey--Glass equation $(q=1)$, we can also infer the existence of a non-trivial invariant measure (i.e., a measure distinct from $\delta_0)$,    even in parameter regimes where the system seems to exhibit  chaos; see Corollary \ref{cor:1.2} and Figure \ref{fig:det:invariant} in {\S}\ref{sec:characterisation}. To the best of the authors' knowledge, these ergodic properties have been noted in, e.g., \cite{lasota1977ergodic,lossom2020density,mitkowski2012ergodic}, as well, but with no rigorous proofs so far. Numerics  in {\S}\ref{sec:characterisation} indicate that solutions to the stochastic Mackey--Glass equation $(q=1)$ are clearly bounded away from zero in probability, implying the existence of a non-trivial invariant measure as well; see also Figures \ref{fig:invariant} and \ref{fig:otherinitial}. The same conclusions can be drawn for the Nicholson's blowflies equation.

\paragraph{Scope and outlook} This paper leaves open many opportunities for further research.
    First of all, an interesting direction would be to investigate when, in the case of $f(0) = 0$, solutions are bounded away from zero in probability; thereby improving the main result in Theorem \ref{main-thrm}.    One way to possibly achieve this is by extending our novel approach in {\S}\ref{Sec5}--{\S}\ref{Sec6}. Another fruitful method could be to utilise the technique in \cite{appleby2003noise,appleby2004oscillation} (see \cite[Lem. 1]{appleby2003noise}, in particular), which studies stochastic oscillations, or perhaps use a combination of both approaches. 

Furthermore,   it is not true  in general that the non-trivial  invariant measures  in
Theorem \ref{main-thrm} are unique; however, uniqueness is expected to hold for the \textit{stochastic} Mackey–Glass equation $(q=1)$, the \textit{stochastic} Nicholson's blowflies equation, and related problems. Uniqueness, as well as  other properties such as, e.g., ergodicity, asymptotic stability/mixing, and absolute continuity with respect to the Lebesgue measure (thus admitting a probability density function) are far from immediate in the setting of SDDEs \cite{butkovsky2017invariant,article:reiss}, as opposed to SDEs and SPDEs \cite{book:asymptotic,book:daprato,da2014stochastic}. This is because, for instance,  the Feller property is not strong. The fact that these properties are not immediate for stochastic nonlinear delayed systems is also reflected in the recent work   \cite[Sec.~6]{mackey2023temporal}, which pursues a more direct Fokker–Planck type approach.






\paragraph{Organisation} This paper is structured as follows. In {\S}\ref{sec:characterisation} we examine, from both an analytical and a numerical perspective, the supports of 
$\mu$ and $\nu$ in Theorem \ref{main-thrm} and elucidate their connection to the system’s long-term dynamics. In {\S}\ref{sec:new:3} we give an overview of the steps leading to the existence of invariant measures. Specifically, we introduce a few definitions, describe the types of Lévy processes that we encounter, and provide general results on the global existence, persistence, and boundedness in probability of solutions. 
In {\S}\ref{Sec5} we derive   one-sided tail bounds|as   in \cite{BosGaaVer1-24}|for It\^o- and L\'evy-driven processes with negative drift, leading to a proof of   Theorem \ref{main-thrm} in {\S}\ref{Sec6} where the problem is essentially divided into a deterministic and stochastic part. 

\paragraph{Acknowledgements} The corresponding author would like to thank Alex Blumenthal and Ioannis (John) Stavroulakis for several helpful insights. Furthermore, the authors are grateful to the anonymous reviewers for their  comments.


\section{Characterisation of   invariant measures and numerics}\label{sec:characterisation}
The existence of $\mu$ and $\nu$ in Theorem \ref{main-thrm} follows from the Krylov--Bogoliubov procedure. Since this construction depends on the choice of initial condition $\Phi$, different initial data may in principle lead to different invariant measures  $\mu$ and $\nu$. Properties of  invariant measures obtained in this way are, generally speaking,  not known a priori. Theorem \ref{main-thrm} should for this reason be appreciated in conjunction with numerical simulations: it provides a means of verifying that the observed structures in, e.g., Figures \ref{fig:det:invariant}--\ref{fig:otherinitial} are not numerical artefacts. In fact, we   know from this theorem  in combination with these illustrations that even much more rich underlying structures exist. In practice,  the Krylov–Bogoliubov method typically constructs the physically “relevant” invariant  measures\footnote{In the literature, Eckmann and Ruelle refer to these as physical measures \cite{eckmann1985ergodic}, since they correspond physically to experimental time averages. Physical measures are essentially the invariant measures that describe the dynamics for a set of initial conditions with positive Lebesgue measure \cite{blumenthal_young_ergodic_theory}.} for negative feedback systems; we explain why this is the case in more detail below.

Time series shown in Figures~\ref{fig:det:invariant}–\ref{fig:otherinitial} are generated using a standard Euler–Maruyama scheme \cite{buckwar2008weak} applied to the stochastic delay equation \eqref{eq:stoch-t-mono-MG2}. Furthermore, the  invariant measures in Figures~\ref{fig:det:invariant}–\ref{fig:otherinitial} are obtained by performing the Krylov--Bogoliubov method numerically.  
 Before we describe our numerical methodology and outline the insights it provides, we give (to a limited extent) an analytical characterisation of the non‑trivial invariant measures arising in Theorem \ref{main-thrm}.  






In the deterministic case ($\sigma=0$), Theorem \ref{main-thrm} combined with \cite[Thm. 3]{article:MGanalysis}|telling us when solutions remain  bounded away from zero for equation \eqref{eq:MG}|yields  the following result. At this point, still not much is known about the support of $\mu$ and whether it is unique.

\begin{corollary}\label{cor:1.2}
    The solution to the deterministic Mackey--Glass equation $(q=1),$
    \begin{equation}
        x'(t)=-\gamma(t)x(t)+r(t)\frac{x(t-\tau)}{1+x(t-\tau )^p},\label{eq:MG}
    \end{equation}
    with initial value $ x_0=\ph$, is bounded away from zero, i.e., $\inf_{t\geq 0}x(t)>0$, provided that
    \begin{equation}
       p>1,\quad\inf_{t\in[-\tau,0]}\ph(t)>0,\quad  \liminf_{t\to 0}\frac{r(t)}{\gamma(t)}>1.
    \end{equation}
    Furthermore, if in addition $\gamma$ and $r$ are constants, then $r>\gamma$ and equation \eqref{eq:MG} admits a non-trivial invariant measure $\nu$ on $C[-\tau,0]$, and hence a stationary distribution $\mu$ with $\mu([m,M])=1$
 for some $m,M>0.$ \end{corollary}

 
Equation \eqref{eq:MG}  has a single steady state at $x=0$  when $r\leq \gamma$ and an additional steady  state at $x=x{}_*$, with $x{}_*^p=(r-\gamma)/\gamma$, when $r>\gamma$, which yields the non-trivial invariant measure $\delta_{x_*}$|no matter the stability of the fixed point $x_*$. From Figure \ref{fig:det:invariant}, for instance, we deduce that there are at least two distinct non-trivial invariant measures for $\tau=1$ and $p=8.$ In fact, we expect for these parameter values to have infinitely many distinct invariant measures, as we believe that there are infinitely many unstable periodic orbits. Naturally, the existence of steady states or limit cycles—either stable or unstable—implies the existence of invariant measures \cite{eckmann1985ergodic}. Typically, we retrieve the same histograms  in Figure \ref{fig:det:invariant} for most other initial conditions.

Let us now proceed  with what we can say analytically about the uniqueness and  support of $\mu$ in Corollary \ref{cor:1.2}.   
A direct consequence of   the results in \cite{berezansky2013mackey,wei2007bifurcation}, see \cite[Thm. 5.1]{wei2007bifurcation} in particular, is the following:
\begin{itemize}
    \item[(i)]   For $r>\gamma $ and $\gamma/r>(p-2)/p$, we have that $x_*$ is asymptotically stable for all $\tau \geq 0,$ and hence there only exists\footnote{Excluding convex combinations of the Dirac measures $\delta_0$ and $\delta_{x_*}.$} the non-trivial invariant measure $\mu=\delta_{x_*};$
    \item[(ii)]  For $r>\gamma$ and  $\gamma/r<(p-2)/p,$ we have that $x_*$ is asymptotically stable $($unstable$)$ when $\tau<\tau_0$ $(\tau>\tau_0)$  with 
    \begin{equation}
        \tau_0=\frac{r}{\gamma\sqrt{[p\gamma-(p-1)r]^2-r^2}}\arccos\left(\frac{r}{p\gamma-(p-1)r}\right).
    \end{equation}
    Hence, for $\tau<\tau_0,$ there is again only the non-trivial invariant measure $\delta_{x_*}$. For $\tau>\tau_0$,    a non-trivial invariant distribution $\mu$ distinct from $\delta_{x_*}$ can exist, and if so, then it satisfies $\mu([m,M])=1$ for some $m,M>0$ with $m\neq M.$
\end{itemize}
    
  \begin{figure}[!t]
    \centering
    \includegraphics[width=0.995\linewidth, trim = 1cm 0 0 0, clip]{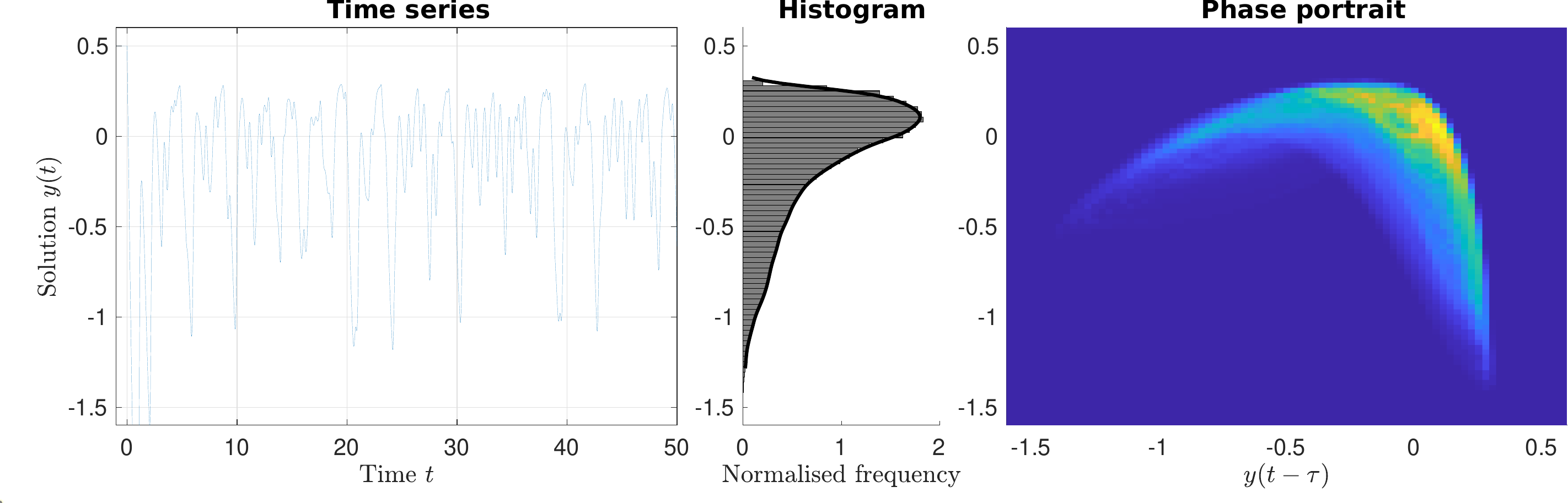} 
    \caption{Simulation of the deterministic equation \eqref{eq:t-mono-MG2} on the time interval $[0,10000]$ with parameters $\gamma=5, r=10, \tau=1,  $ together with $f(x)=x(1+x^p)^{-1}$ where   $p=8$. Initial data is given by   $\psi(t)=0.5,t\in[-1,0].$ On the left, we plot the solution on the time interval $[-1,50]$. In the middle panel, we show a histogram of all values over the interval $[5000,10000]$, to avoid transient effects from the initial condition. This histogram approximates the stationary distribution of the system related to the strange attractor (whose existence has not been proved analytically). On the right, we display the phase portrait as a heat map, which consequently visualises the push-forward of an invariant measure $\nu$ on $C[-\tau,0]$ under the evaluation map $C[-\tau,0]\to\mathbb R, \varphi\mapsto (\varphi(-\tau),\varphi(0))$. Observe that changing the interval $[5000,10000]$ into $[\alpha,\beta]$ for any $\beta\gg\alpha\gg 0$  leads (approximately) to the same figures, indicating that the structures are indeed invariant.}
    \label{fig:det:invariant}
\end{figure}


Note that $\mu$ in (ii) cannot be another point measure, because if it were, then we should have had another fixed point in \eqref{eq:MG}. While for $\tau>\tau_0$ the fixed point   $x_*$ is unstable, we caution the reader that theoretically one could still construct the measure $\delta_{x_*}$ by means of the Krylov--Bogoliubov procedure, even if the initial value is unequal
to $x_*.$\footnote{Such phenomena already occur in ODEs. For example, consider   a continuous-time dynamical system in $\mathbb R^2$ with a ``figure eight'' separatrix in the phase plane, where  the intersection is a saddle node at the origin, with two unstable nodes inside each homoclinic orbit \cite[Fig. 17]{eckmann1985ergodic}; see \cite[eq. (4.46)]{meiss2007differential} for a concrete set of equations. For any initial condition unequal to the two unstable nodes, the limit of $\mu_T$|see \eqref{eq:muT}|converges to $\delta_0$, implying that all non-trivial solutions spend most of their time near the origin (as that is where the slow dynamics happen). An  example for which $(\mu_T)_{T\geq 0}$ admits convergent subsequences, but no unique limit, can be as simple as a Lotka--Volterra system with 3  species. This system can exhibit non-periodic population oscillations of bounded amplitude but ever increasing cycle time \cite{may1975nonlinear}.} 
 This however turns out not to be the case if we  perform the procedure  numerically; see, e.g., Figure \ref{fig:det:invariant}. Furthermore, in \cite{wei2007bifurcation} it is shown that there is a Hopf bifurcation at $x=x_*$ when $\tau=\tau_0$ and that the
    the system can possibly undergo  multiple Hopf bifurcations; this follows from the equation  \eqref{eq:characteristic:CH3} below but with $0$ replaced by $x_*$. 
      Analytically, the global existence of periodic solutions are known under (ii) with $\tau>\tau_0$ and the additional, sufficient but not necessary, condition 
 \begin{equation}
     \max\{1,(p-1)^2/4p\}<\sqrt{2}\gamma/r;\label{eq:cond:max}
 \end{equation}
 see \cite[Thm. 5.2]{wei2007bifurcation}.
 We point out that   condition \eqref{eq:cond:max} is not satisfied in Figures \ref{fig:det:invariant}--\ref{fig:otherinitial}. Similar results can be found in \cite{wei2007bifurcation}   for the Nicholson's blowflies equation and for a general class of functionals $f$.


Multiplicative noise as in \eqref{eq:delay-noise} typically ensures that  $x=0$ remains a steady state, implying that  $\delta_0$ is an invariant measure even in the setting with noise. 
For $r>\gamma>0$ and  sufficiently small $\sigma^2>0$, we believe  an invariant measure $\mu$ exists   which is concentrated near the value $x_\ast$ for the Mackey--Glass functional \eqref{eq:Mackey-Glass} with $q=1$, or at least  $\mu(\{0\}) = 0$  and thus $\mu\neq \delta_0;$ see also Figure \ref{fig:invariant}.    This conjecture can  be extended to the  general setting   where $f(0)=0$ and  $f'(0)>0.$ The characteristic equation associated with $x=0$ reads  
\begin{equation}
    \lambda + \gamma=rf'(0)e^{-\lambda \tau},\quad \lambda \in\mathbb C,\label{eq:characteristic:CH3}
\end{equation}
where explicit solutions are given by $\lambda_m = -\tau^{-1}W_m(-\tau\theta),$ with $\theta = rf'(0)-\gamma$ and  $W_m$, $m\in\mathbb Z$, denoting the $m$-th branch of the Lambert function.
The leading eigenvalue is $\lambda_0$, see \cite{corless1996lambert}, and the sign of the real part of $\lambda_0$ coincides with the sign of $\theta.$ Therefore, the steady state $x=0$ is (locally) asymptotically stable if $rf'(0)<\gamma$ and  unstable if $rf'(0)>\gamma$ holds \cite[Thm. 6.8]{diekmann2012delay}. Without noise and  for sufficiently small $\sigma^2>0$,  we believe that the condition $rf'(0)>\gamma$ is  a good indication for the existence of a non-trivial invariant measure  in  models of the form \eqref{eq:delay-noise}. Similarly, conclusions such as (i) and (ii) above can be extended to the noisy setting.
The discussion above seems to suggest that the probability of extinction under  sufficiently small noise is zero when $x=0$ is unstable. Take note that this probability can become strictly positive if one replaces the multiplicative noise term in \eqref{eq:delay-noise} by a non-locally Lipschitz coefficient \cite{Busse}.

 \begin{figure}[!t]
    \centering
    \includegraphics[width=0.995\linewidth]{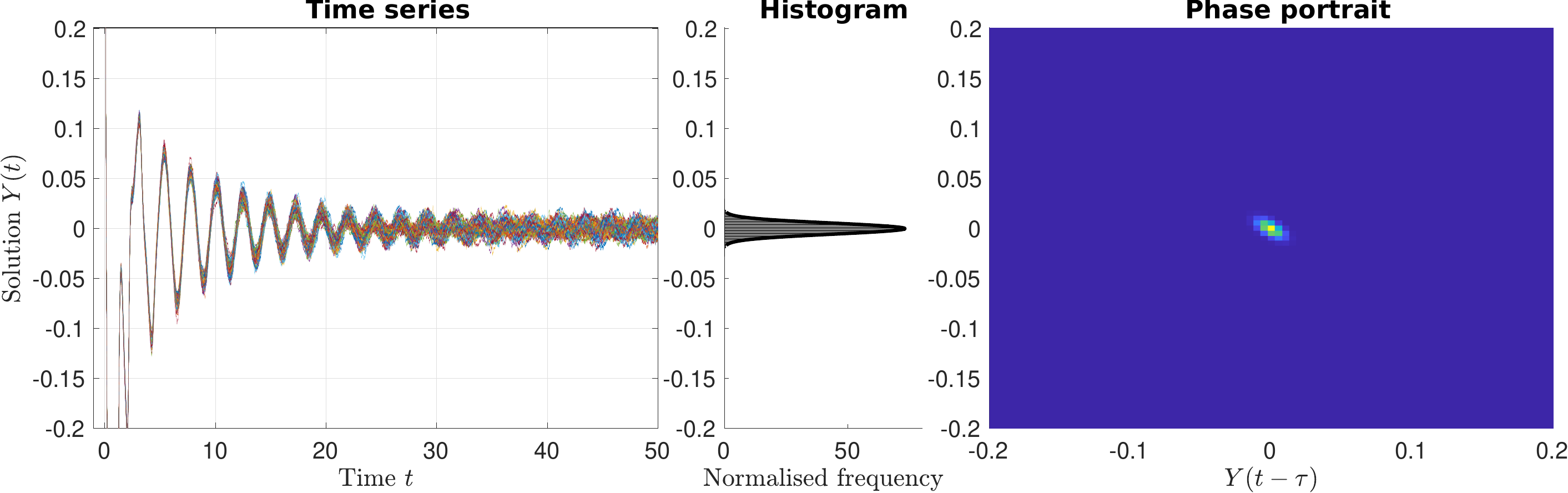}\\[.25cm]
    \includegraphics[width=0.995\linewidth]{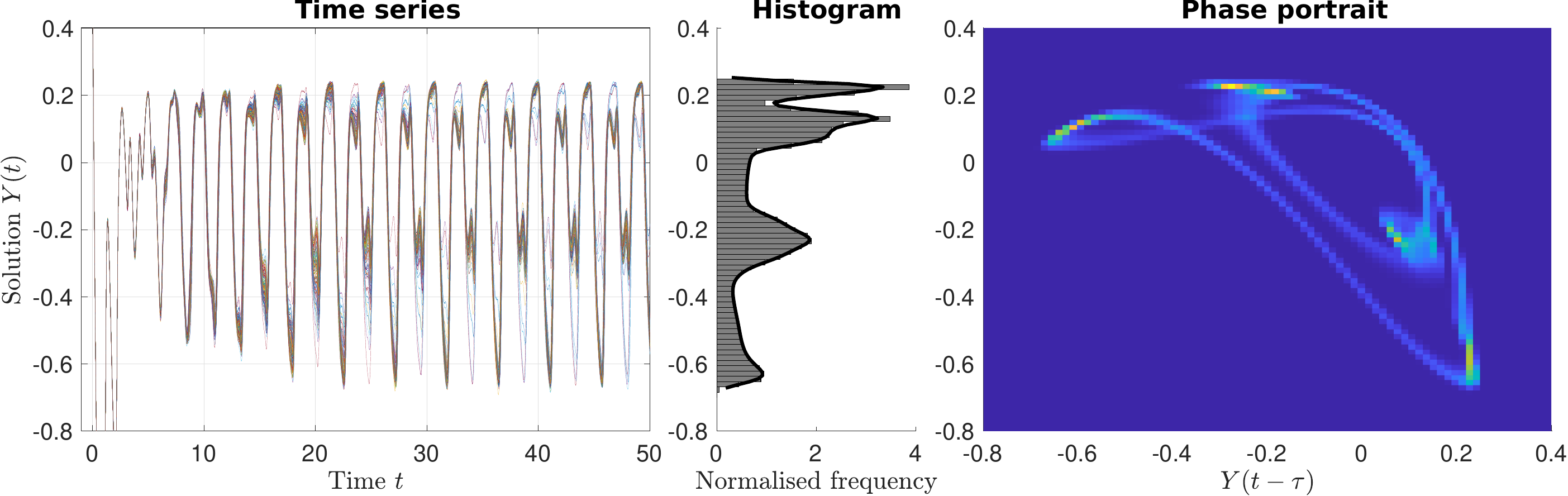}\\[.25cm]
    \includegraphics[width=0.995\linewidth]{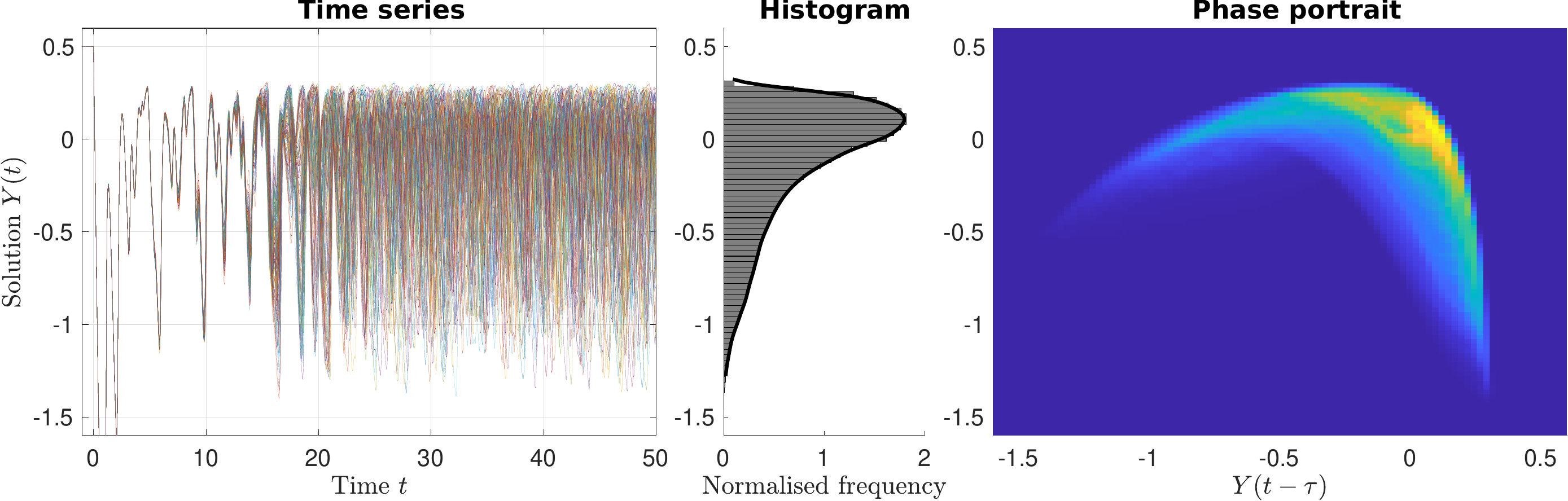}
    \caption{Simulation of 100 sample paths of the solution to equation \eqref{eq:stoch-t-mono-MG2} on the time interval $[0,500]$ with parameters $\gamma=5, r=10, \tau=1, b=0.01, a=-b^2/2,$ together with $f(x)=x(1+x^p)^{-1}$ where $p=4$ (first row), $p=6$ (second row), and $p=8$ (third row). Initial data is given by $Y_0=\psi,$ with $\psi(t)=0.5,t\in[-1,0].$ On the left, we plot the 100 realisations on the time interval $[-1,50]$. In the middle panel, we show a histogram of all values over the interval $[250,500]$, to avoid transient effects from the initial condition. On the right, we provide a heatmap of the phase portrait of all sample paths on the interval $[250,500]$. Observe that the histogram and phase portrait for $p=8$ is similar to that of the deterministic case; see Figure \ref{fig:det:invariant}.}
    \label{fig:invariant}
\end{figure}
 \begin{figure}[!t]
    \centering
    \includegraphics[width=0.995\linewidth]{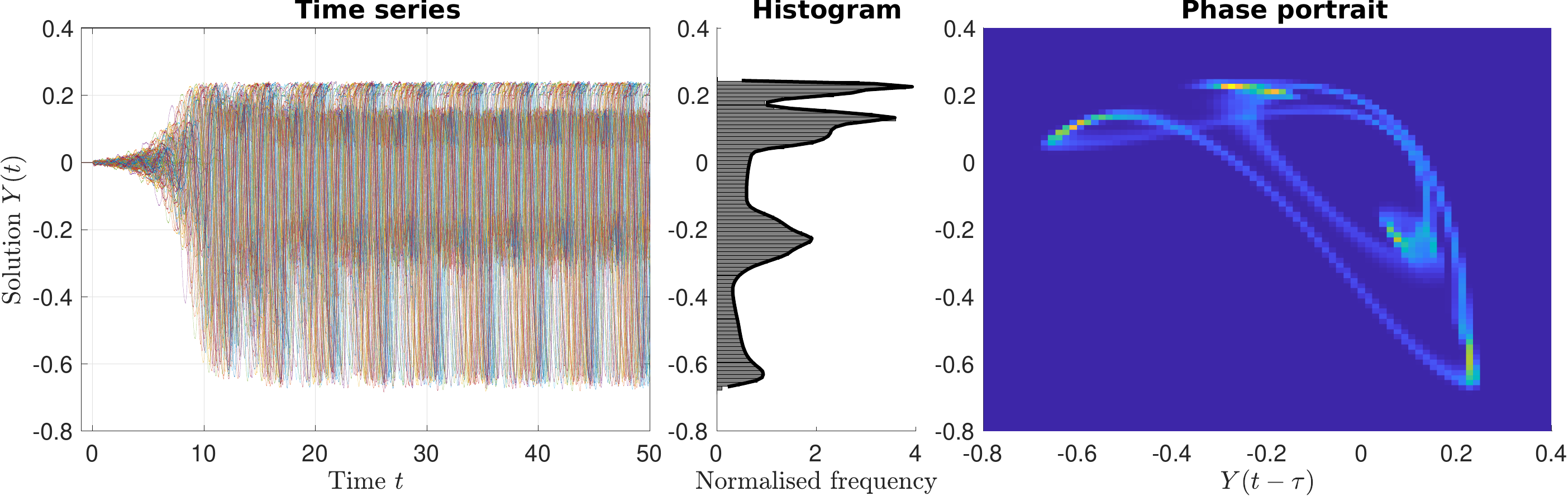} 
    \caption{The same as in Figure \ref{fig:invariant} for $p=6$, but now with $\psi(t)=0, t\in[-1,0].$ This figure both motivates the uniqueness of the invariant measure in the stochastic setting and illustrates that, although time series data can sometimes appear quite irregular,   visualisations of invariant measures  clearly reveal  the underlying statistical structure.}
    \label{fig:otherinitial}
\end{figure}


Ultimately, 
let us  discuss the obtained illustrations of the invariant measures in Figures \ref{fig:det:invariant}--\ref{fig:otherinitial} and their implications in more detail.  To this end, write $Y(t)$ for a solution to the transformed equation \eqref{eq:stoch-t-mono-MG2},  with $b=\sigma\in\mathbb R$ and $a=-\frac12b^2$. For $0\leq \alpha \ll T<\infty $, define
\begin{equation}
\mu_T(A) := \frac{1}{T} \int_\alpha^{\alpha+T} \mathbb{P}(Y(t) \in A)\, \mathrm  dt,\qquad \nu_T(B) := \frac{1}{T} \int_\alpha^{\alpha+T} \mathbb{P}(Y_t \in B)\, \mathrm  dt,\label{eq:muT}
\end{equation}
for measurable sets $A\subset \mathbb R$ and $B\subset C[-\tau,0].$ Here,  $\mu_T$ and $\nu_T$  are probability measures on $\mathbb R$ and $C[-\tau,0]$, respectively. 
  The Krylov--Bogoliubov theorem \cite[App. C]{artikel1-general}  allows us to conclude that $(\nu_T)_{T\geq 0}$ has a convergent subsequence, whose limit we denote by $\nu$. Hence, the invariant measures we construct describe the system's dynamics over a long-time horizon. In many applications,  the limit   $  \nu = \lim_{T\to \infty} \nu_{T}$ simply exists, which is then an  invariant measure.  Consequently, the existence of this (subsequential) limit $\nu$   implies the existence of a stationary distribution $  \mu = \lim_{T \to \infty} \mu_T$ satisfying $\mu(A)=\nu(\{\varphi\in C[-\tau,0]:\varphi(0)\in A\})$.    

 In view of the above,  one obtains   a numerical approximation of a stationary distribution $\mu$ using discrete sums:
\begin{equation}
\mu(A) \approx \frac{1}{(M+1)N} \sum_{k=0}^M 
\sum_{i=1}^N \mathbf{1}_A(y_i(t_k)),\qquad \mathbf 1_A(y)=\begin{cases}
    1&y\in A,\\ 0&y\not\in A,
\end{cases}\label{CH3:eq:i:histogram}
\end{equation}
where $  y_i(t_k)$  denotes the $i$-th trajectory at time $t_k=\alpha+k\Delta t$, on a uniform grid of length $T=M\Delta t$ with step size $\Delta t$, and where $N$ denotes the number of trajectories. This discrete formulation directly corresponds to creating a histogram as in Figures \ref{fig:det:invariant}--\ref{fig:otherinitial}. To ensure that \eqref{CH3:eq:i:histogram} provides a good approximation of the stationary distribution, $N$ and $M$ should initially be chosen sufficiently large, and the resulting histogram should remain stable under variations in $N$ and $M$. By doing so, we tacitly exploit the fact that the limit $\nu=\lim _{T\to\infty} \nu_T$ exists (and thus not only as a subsequential limit). In numerous situations, a single trajectory ($N = 1$) suffices to approximate the invariant measure, indicating  the  measure is ergodic. Indeed, the same one- and two-dimensional histograms in Figures \ref{fig:invariant} and \ref{fig:otherinitial} are obtained when  a single trajectory is considered, which suggests the underlying invariant measures are ergodic.
The phase portraits within these figures are a numerical approximation of the push-forward measure $\mu'$ of $\nu$ under the evaluation map $\varphi\to(\varphi(-\tau),\varphi(0)),$ i.e., $\mu'(A')=\nu(\{\varphi\in C[-\tau,0]:(\varphi(-\tau),\varphi(0))\in A'\})$ for   $A'\subset \mathbb R^2.$ Here, $\mu$ and $\mu'$ are finite-dimensional projections of the infinite-dimensional object $\nu.$


In the presence of noise, trajectories are continuously pushed away from unstable orbits, preventing them from lingering near such structures. This regularising effect strongly suggests the uniqueness of a non-trivial invariant measure for the stochastic Mackey--Glass equation $(q=1)$, even in chaotic regimes where the system already exhibits substantial mixing. In the deterministic setting \eqref{eq:MG}, there can generally be more non-trivial invariant measures than just $\delta_{x_*}$. In fact, there are  parameter value regions for which we have co-existing stable limit cycles \cite{losson1993solution,tarigo2022characterizing}, which then leads to multiple invariant measures (and hence infinitely many, by taking convex combinations) associated with attractors.   The addition of noise is expected to eliminate unstable invariant structures and to induce metastable switching between basins of attraction. One therefore expects the stochastic counterpart \eqref{eq:delay-noise} with $f(x)=x(1+x^p)^{-1}$ to possess a unique non-trivial invariant measure.

 As mentioned previously,  “unstable’’ invariant measures can be constructed through the Krylov–\linebreak Bogoliubov method.   On the other hand, these are rarely observed numerically, as one typically detects the presence of a “stable” invariant measure  more easily. To illustrate this, note that the Dirac measure $\delta_{x_*}$ can be retrieved numerically through \eqref{CH3:eq:i:histogram} by taking the solution with $\psi=\log(x_*)$ as initial condition. However, for $p=6$, for instance, this equilibrium  $x_*$ loses stability to a periodic solution, hence   the system admits a non-trivial invariant measure supported on that periodic orbit. This invariant measure  can be numerically obtained through \eqref{CH3:eq:i:histogram} again, now by taking any $\psi\neq \log(x_*)$ as initial condition. In a similar fashion, we find in Figure \ref{fig:det:invariant} (with $p=8)$ some  invariant measure, which is also obtained for most other initial data. This suggests that it is   connected to the strange attractor (whose existence has not been proved analytically).

\section{Towards the existence of invariant measures}\label{sec:new:3}
We obtain  the existence of a non-trivial invariant measure by applying a procedure thanks to  Krylov and Bogoliubov \cite{kryloff1937theorie} on equation \eqref{eq:MGeq4}; see \cite[App C.]{artikel1-general} for the general existence theorem(s) in the delay setting.  Note that applying the Krylov--Bogoliubov method   directly on the original equation \eqref{eq:delay-noise}  also results into the existence of an invariant measure; even in the case  $f(0)=0.$ However, one does  not exclude the possibility now of having found the  Dirac measure $\delta_0$. 

Most of the necessary technical groundwork is already carried out in \cite{artikel1-general}, allowing us to avoid a lot  of the abstract machinery in this paper. In particular, we have proved  the proposition  below, which  tells us that if a solution  is bounded in probability, then we can construct from this solution an invariant measure $\nu$, hence a stationary distribution $\mu,$ as discussed in {\S}\ref{sec:characterisation}. Technically speaking, this is because we have proved  the  boundedness in probability of $(X(t))_{t\geq 0}$   implies that the segment process $\big(\|X_t\|_\infty\big)_{t\geq 0}$ is bounded in probability and is  tight; see \cite{artikel1-general} for more information.

\begin{proposition}[Corollary 4.11 of \cite{artikel1-general}]\label{cor:application}Suppose that $f:\mathbb R\to \mathbb R$ is  locally Lipschitz continuous, positive on $(0,\infty)$, and bounded from above.  In addition, assume    $\gamma,r>0$.  	Consider the    stochastic delay differential equation
 	\begin{equation}
 		\mathrm d Y(t)=\left[-\gamma+re^{-Y(t)}f\big(e^{Y(t-1)}\big)\right] \mathrm d t+a\left(Y_{t}\right) \mathrm d t+b\left(Y_{t-}\right) \mathrm d L(t),\label{eq:MGeq4}
 	\end{equation}
 	where $a,b$ are  time-independent and proper locally Lipschitz. Further, assume $L=(L(t))_{t\geq 0} $ is an integrable Lévy process.  	Let  $\alpha,\beta\geq 0$ be  non-negative constants   such that
 	\begin{equation}
 		|a(\varphi)| \leq \alpha \quad \text { and } \quad b(\varphi)^{2} \leq \beta^{2}, \quad \text { for all } \varphi \in D[-1,0].
 	\end{equation}
Finally, assume that for every initial process  $\Phi$
the corresponding solution exists globally. If one of these   solutions   is bounded   in probability,
  then there exists an invariant measure and hence a stationary solution. 
\end{proposition}

In {\S}\ref{sec:typesnoise} we clarify the notions of stationary distributions and invariant measures, give a definition of boundedness in probability, and also list the types of Lévy noise encountered in this paper. In {\S}\ref{sec:glb} we consider conditions under which solutions to \eqref{eq:delay-noise} remain positive and exist globally, and conditions for global existence of solutions to \eqref{eq:MGeq4}, followed by a discussion of their boundedness in probability in {\S}\ref{sec:bndness}.



\subsection{List of definitions and  types of L\'evy noise}\label{sec:typesnoise}
Consider the generic autonomous initial value problem of the form 
  \begin{equation}\label{eq:SDDE-main}\left\{
	\begin{array}{rlll}
		\mathrm d X(t)&=&a(X_t)\,\mathrm d t+b(X_{t-}) \,\mathrm d L(t), &\quad \text { for } t\geq 0, \\[.05cm]
		X(u)&=&\Phi(u),& \quad \text { for } u \in[-\tau, 0],
	\end{array}\right.
\end{equation}
where we assume $\Phi\in\mathbb D[-\tau,0]$, i.e.,  a stochastic process taking values in $D[-\tau,0],$  $L=(L(t))_{t\geq 0}$ is a L\'evy process, and  $a,b:D[-\tau,0]\to\mathbb R$ are locally Lipschitz with  respect to the supremum norm $\|\cdot\|_\infty$ and are, in addition, assumed  \textit{proper} for well-posedness; see \cite[Sec. 2]{artikel1-general} for more details on the existence and uniqueness of solutions to \eqref{eq:SDDE-main}.


\begin{definition}\label{def:invariant}
	A solution $X=(X(t))_{-\tau\leq t<\infty}$ to  problem \eqref{eq:SDDE-main} with maximal existence time, i.e., $T_\infty=\infty$ $\mathbb P$-a.s., is called \texttt{stationary} if the probability distribution of $X(t)$ coincides, for all $t \geq 0$, with the probability distribution of $\Phi(0)$. In that  case, the probability distribution of $\Phi(0)$ is  said to be a \texttt{stationary distribution} of the  delay equation in \eqref{eq:SDDE-main}.
\end{definition}


\begin{definition}\label{def:invariantm}
A Borel probability measure $\nu$ on $E[-\tau,0]$\footnote{We allow $E\in\{C,D\}.$ The space $C[-\tau,0]$ is to be endowed with the uniform topology, as usual, but the space $D[-\tau,0]$ must be endowed with the {Skorokhod} topology; see \cite{artikel1-general} for more information. For the intents and purposes of this paper, we do not go further into these technicalities. } is an \texttt{invariant} \texttt{measure} of the delay  equation in \eqref{eq:SDDE-main} if the segment process $(X_t)_{t \geq 0}$, with initial condition $X_0 = \Phi$ distributed according to $\nu$, has the same distribution $\nu$ at every time $t \geq 0$. The push-forward measure of $\nu$ under the evaluation map
	\begin{equation}
 E[-\tau,0] \rightarrow \mathbb{R}:	\varphi \mapsto \varphi(0)
	\end{equation}
	is then a stationary distribution.
\end{definition}
Note that a stationary distribution is an invariant measure on $\mathbb{R}$. Also, 
an invariant measure (on $E[-\tau,0]$) contains richer information on the dynamical system than a stationary distribution does. While the terms ``stationary distribution'' and ``invariant measure'' are often used interchangeably throughout the literature, we  distinguish these notions in this paper for clarity.


\begin{definition}Let $I$ be some index set, e.g., take $I=\mathbb N$ or $I=\Rplus$. A family
$(Z_{\eta})_{\eta\in I}$ of real-valued random variables is    \texttt{bounded\,\,above {\normalsize\textnormal{(resp.,}} below{\normalsize\textnormal)} in\,\,probability}  if for every $\varepsilon>0$ there exists a real number $M_{\varepsilon} \in \mathbb{R}$ such that for all $\eta \in I$ we have 
 \begin{equation}
 \mathbb{P}\left(Z_{\eta} >  M_{\varepsilon}\right) < \varepsilon \quad\quad\quad\quad\quad\quad\quad\big(\text {resp., }\mathbb{P}\left(Z_{\eta} <  M_{\varepsilon}\right) <  \varepsilon \big).\quad
\end{equation}
 Differently put, we have $(Z_{\eta})_{\eta\in I}$ is bounded above (resp., below) in probability if and only if
 \begin{equation}
 	\lim_{R\to\infty}\sup_{\eta\in I}  \mathbb{P}\left(Z_{\eta} \geq  R\right)=0\quad\quad\quad\left(\text {resp., }	\lim_{R\to\infty}\sup_{\eta\in I}  \mathbb{P}\left(Z_{\eta} \leq -R\right)=0 \right).
 \end{equation}
A real-valued family $(Z_{\eta})_{\eta\in I}$ is said to be \texttt{bounded\,\,in\,\,probability} if it is both bounded above and below in probability. That is, for every $\varepsilon>0$ there exists a real number $M_{\varepsilon} \in \mathbb{R}$ such that for all $\eta \in I$ we have 
$
\mathbb{P}\left(|Z_{\eta}| >  M_{\varepsilon}\right) < \varepsilon,
$
or, in short,
$
\lim_{R\to\infty}\sup_{\eta\in I}  \mathbb{P}\left(|Z_{\eta}| \geq  R\right)=0.
$
\end{definition}
Note that  being bounded in probability is weaker than being   $\mathbb P$-a.s.\ bounded. For a more detailed discussion on these definitions above, we refer to \cite{artikel1-general}.


 








\paragraph{Types of Lévy noise} Throughout this paper, we consider SDDEs driven by what call  ``regulated Lévy processes''. To this end, we first introduce the following class of Lévy processes and introduce some notation that is also used in {\S}\ref{Sec5}--{\S}\ref{Sec6}. 

\begin{itemize}\item[] \begin{itemize}
 \item[(HJudi)] The process $L=(L(t))_{t\geq 0}$ is a square integrable L\'evy process that is of finite intensity, i.e., $\nu(\mathbb R)<\infty$, where $\nu$ is the associated L\'evy measure.
\end{itemize}\end{itemize}  

The finite intensity property above implies that  $L$ has a finite number of jumps on  any compact time interval \cite[Thm. 21.3]{book:sato}. Either \cite[Thm. 2.3.9]{book:applebaum} or \cite[Lem. 2.8]{book:kyp} subsequently tells us that  any $L$ of type (HJudi)  is a \texttt{jump\,\,diffusion\,\,process}. That is, a sum of two independent processes: a Brownian motion $W=(W(t))_{t\geq 0}$ which is scaled with the dispersion coefficient $\sigma^2$ and includes a drift; and a compound Poisson process $Z=(Z(t))_{t\geq 0}$   with jump measure $\frac1{\nu(\mathbb R)}\nu.$ 
Indeed, we have
\begin{equation}
    L(t)=\left(\gamma - \int_{\{|x|\leq 1\}}x\nu(\mathrm dx)\right)t+\sigma W(t)+Z(t),\quad Z(t)=\sum_{k=1}^{N(t)}Z_k,\quad t\geq 0.\label{eq:not_canon}
\end{equation}
We say  $N=(N(t))_{t\geq 0}$  is a Poisson process associated to the L\'evy process $L$.       
Rewriting yields
\begin{equation}
    L(t)=\left(\gamma + \int_{|x|>1}x\nu(\mathrm dx)\right)t+\sigma W(t)+\big[Z(t)-\lambda_N \mathbb EZ_1t\big],\quad t\geq 0.
\end{equation}
In particular, we have $\lambda_N \mathbb EZ_1=\int_{\mathbb R}x\nu(\mathrm dx).$ A L\'evy process of finite intensity is a martingale if and only if $\gamma = -\int_{\{|x|>1\}}x\nu(\mathrm dx).$ Suppose  $Z_1$ is centred, i.e.,  $\mathbb EZ_1=0,$ then we have  $\gamma = \int_{\{|x|\leq 1\}}x\nu(\mathrm dx)$ and  $Z$ is a compound Poisson process.

Note that  $\mathbb EZ_1^2<\infty$ holds if and only if $L$ is of class (HJudi), and then the quadratic variation  and the predictable quadratic variation (whether $L$ is a martingale or not) equal
\begin{equation}
    [L](t)=\sigma^2t+\sum_{s\leq t}(\Delta Z(t))^2=\sigma^2t+\sum_{k=1}^{N(t)}Z_k^2\quad\text{and}\quad\langle L\rangle(t)=\lambda t,\quad t\geq 0,
\end{equation}
respectively, where $\lambda =\sigma^2+\lambda_N\mathbb EZ_1^2$ and  $\lambda_N=\mathbb E [N(1)]$  is the intensity of the Poisson process $N$.


\begin{itemize}\item[] \begin{itemize}
 \item[(HReg)] The process $L=(L(t))_{t\geq 0}$ is  a L\'evy process of class (HJudi)  and   satisfies the following two additional properties:
  \begin{itemize}\item[] \begin{itemize}
     \item[(P1)] the process experiences no continuous drift, i.e., $\gamma = \int_{\{|x|\leq 1\}}x\nu(\mathrm dx);$
      \item[(P2)] jumps are $\mathbb P$-a.s.\ uniformly bounded by some $\zeta\geq 0,$ i.e., we have $|\Delta L(t)|\leq \zeta $  $\mathbb P$-a.s., where $\Delta L(t)=L(t-)-L(t)$ and $L(t-)=\lim_{s\nearrow t}L(s).$
 \end{itemize}\end{itemize}
 Such processes are referred to as \texttt{regulated\,\,L\'evy\,\,processes} in this paper.
 \end{itemize}\end{itemize}
  \begin{itemize}\item[] \begin{itemize}
 \item[(HRegM)] The process $L=(L(t))_{t\geq 0}$ is a regulated L\'evy process and a martingale,
 hence called a \texttt{regulated\,\,L\'evy\,\,martingale}.
\end{itemize}\end{itemize} 
Property (P2) implies that $|Z_1|\leq \zeta$ holds $\mathbb P$-a.s., which enables us to keep jumps under control in a   pathwise manner; this is an essential property for our estimates in   {\S}\ref{Sec5}--{\S}\ref{Sec6}.  Furthermore, observe that 
 $\lambda=\sigma^2+\lambda_N\mathbb EZ_1^2\leq \sigma^2+\lambda_N\zeta^2$. Property (P1) is not really important, since it will only keep certain expressions simple.  
 Finally, note that (HRegM) implies $\mathbb EZ_1=0$ or $\lambda_N=0$. The latter  means that $L(t)=\sigma W(t), $ $t\geq 0,$ is a scalar multiple of a Brownian motion.

\subsection{Global existence and  the permanence property}\label{sec:glb} Global existence and the permanence property of solutions to the stochastic Wright's equation is relatively straightforward \cite[Prop. 1.1]{BosGaaVer1-24},
 yet it is unclear how to apply the same proof strategy to delay equations with stochastic negative feedback.
Note that the permanence property of \eqref{eq:delay-noise}  can be inferred by considering the transformed equation \eqref{eq:stoch-t-mono-MG2} and  invoking \cite[Prop. 3.4]{artikel1-general}. 

When $f$ is given by, e.g., \eqref{eq:Mackey-Glass}, global existence of solutions follows from the fact that $f$ is globally Lipschitz on $\mathbb R$|upon smooth extension. For $f$ as in Theorem \ref{main-thrm}, it can either be deduced from the fact that for $X\geq 0$ all the terms in  \eqref{eq:delay-noise} are of linear growth, see \cite[Sec. 2]{artikel1-general},  but it  also follows from the novel approach in {\S}\ref{Sec5}--{\S}\ref{Sec6}. These conclusions  hold for certain types of L\'evy noise  as well.

\subsection{Boundedness in probability}\label{sec:bndness}
 
 The approach in {\S}\ref{Sec5}--{\S}\ref{Sec6}, which involves stochastic estimates applied to pathwise bounds obtained from tracking  trajectories, has successfully shown that most solutions are bounded away from zero in probability  when 
$f(0)>0$. 
Recall that the same method has also  been effective for the stochastic Wright's equation \cite{BosGaaVer1-24}. In  case  $f(0) = 0$, if one could prove that there exists a solution to \eqref{eq:stoch-t-mono-MG2} which is bounded below in probability|as we were able to show when $f(0) > 0$|then the corresponding solution to \eqref{eq:delay-noise} would be bounded away from zero in probability.  There is clear numerical evidence that solutions do not go extinct, i.e., solutions are bounded away from zero $\mathbb P$-a.s., in the stochastic Mackey--Glass equations  with $r>\gamma$ and $\sigma>0$ in \eqref{eq:delay-noise} sufficiently small, provided that the initial condition is positive of course. This particularly allows us to infer from  Theorem \ref{main-thrm} the existence of a non-trivial invariant measure.
 Although numerically it seems evident that solutions do not go extinct for certain parameter regimes  when  $f(0)=0$, rigorously proving that then there is at least one solution that is bounded away in probability  remains an open problem. 
 
It is worth pointing out that the bound from above in probability can actually be shown without much difficulty. We can either do this directly, as demonstrated in the proof of Proposition  \ref{prop:alternatief}, or by means of the approach in {\S}\ref{Sec5}--{\S}\ref{Sec6}. The proof below is inspired by \cite{wang2019stochastic}; estimate  \eqref{eq:ultimate} shows that the solution is ultimately bounded in mean.

\begin{proposition}\label{prop:alternatief}
Suppose  $f:\mathbb R\to \mathbb R$ is  locally Lipschitz continuous, non-negative on $(0,\infty)$, and bounded from above. Assume  $\tilde \gamma=\liminf_{t\to \infty}\gamma(t)>0 $, $\tilde r=\sup_{t\geq 0}r(t)<\infty$, and  $g:C[-\tau,0]\to \mathbb R$  locally Lipschitz  with respect to the supremum norm $\|\cdot\|_\infty.$ If the solution $(X(t))_{-\tau\leq t<\infty }$ to 
\begin{equation}
    \mathrm dX(t)=[-\gamma(t)  X(t)+r(t)f(X(t-\tau))]\mathrm dt+\sigma g(X_t)\mathrm dW(t),\quad X_0=\Phi,
\end{equation}
 for some non-negative $\mathcal F_0$-measurable random variable $\Phi$ taking values in $C[-\tau,0]$ together with $\mathbb{E}[\Phi(0)]<\infty$, exists globally and remains non-negative, then $(X(t))_{t\geq 0}$ is bounded in probability.
\end{proposition}
\begin{proof}
    Define the continuous function $\Gamma(t)=\int_0^t\gamma(s)\mathrm ds$ together with the  process $Y(t)=e^{\Gamma(t)}X(t).$ An application of  It\^o's formula gives us
    \begin{equation}
        \mathrm dY(t)=\mathrm d[e^{\Gamma(t)}]X(t)+e^{\Gamma(t)}\mathrm dX(t)=r(t)e^{\Gamma(t)} f(X(t-\tau))+\sigma e^{\Gamma(t)} g(X_t)\,\mathrm dW(t),
    \end{equation}
    hence
    \begin{equation}
        X(t)=e^{-\Gamma(t)}X(0)+e^{-\Gamma(t)}\int_0^tr(s)e^{\Gamma(s)}f(X(s-\tau))\,\mathrm ds+\sigma e^{-\Gamma(t)}\int_0^tg(X_s)\,\mathrm dW(s).
    \end{equation}
    Taking expectations on both sides yields
    \begin{equation}
    \begin{aligned}\label{eq:EX-Theta}
        \mathbb E[X(t)]&=e^{-\Gamma(t)}\mathbb E[X(0)]+e^{-\Gamma(t)}\mathbb E\left[\int_0^tr(s)e^{\mu(s)}f(X(s-1))\,\mathrm ds\right]\\
        &\leq e^{-\Gamma(t)}\mathbb E[X(0)]+\tilde rMe^{-\Gamma(t)}\int_0^te^{\Gamma(s)}\mathrm ds =: \Xi(t).
    \end{aligned}
    \end{equation}
Thanks to an application of l'H\^opital's rule, we obtain  
    \begin{equation}
        \limsup_{t\to\infty}\mathbb E[X(t)]\leq \limsup_{t\to\infty}\Xi(t)\leq \dfrac{\tilde rM}{\liminf_{t\to\infty}\gamma(t)}\leq \dfrac{\tilde rM}{\tilde \gamma},\label{eq:ultimate}
    \end{equation}
from which we can conclude that $(X(t))_{-\tau\leq t<\infty}$ is bounded in probability. Indeed, from Markov's inequality and \eqref{eq:EX-Theta} we deduce    \begin{equation} 
        \lim_{R\to\infty}\sup_{t\geq 0}\mathbb P(X(t)>R)\leq \lim_{R\to\infty}\sup_{t\geq 0}R^{-1}{\Xi(t)} =0,
    \end{equation}
    since $
        \sup_{t\geq 0} \Xi(t)\leq \sup_{0\leq t\leq T}  \Xi(t)+\limsup_{t\to 0}  \Xi(t)\leq K+\tilde rM/\tilde \gamma, $
    for some $T>0$ sufficiently large and  constant $K\geq0.$ The latter exploits continuity   of $t\mapsto \Xi(t).$  
\end{proof}


Note that the Brownian motion $W$ in Proposition \ref{prop:alternatief} can be replaced by any (local) martingale $M$, as the proof remains valid (after additional stopping time arguments). Moreover, setting $g(\ph)=\ph(0)c(\ph),$ $\ph\in C[-\tau,0]$, gives rise to model  \eqref{eq:delay-noise}.

\section{It\^o- and L\'evy-driven processes with negative drift}\label{Sec5} 
\noindent Let $(\Omega,\mathcal F,\mathbb F, \mathbb P)$ be a filtered probability space satisfying the usual conditions. In this section, we derive probability estimates for  Lévy-driven processes  with negative drift. That is, we consider
\begin{equation}
  	Y(t)=-\dint0t{a(s)}s+\dint{0}{t}{b(s)}L(s), \quad t \geq 0,\label{eq:Y(t)}
  \end{equation}
 where $L=(L(t))_{t\geq 0}$ is a regulated L\'evy process, relevant for our applications. Suitable conditions for the integrands $(a(t))_{t \geq 0}$ and $(b(t))_{t \geq 0}$  will be discussed below. 
    Recall $Y\in\mathbb D[0,\infty)$, i.e., $Y$ is an $\mathbb F$-adapted process with  \caglad sample paths ($\mathbb P$-a.s.\ almost) everywhere, and  that this  process has continuous sample paths whenever the integrator does not admit jump events.

        The ultimate purpose
        of this section is to show that the 
   \texttt{reverse\,\,time\,\,supremum}  process\footnote{Measurability of the reverse time supremum  at any instant $t$ is  due to the sufficient regularity of $Y$. }
  \begin{equation}
  	\left(\sup _{0 \leq \theta \leq t}(Y(t)-Y(\theta))\right)_{t \geq 0}\label{eq:reverse-time-sup}
  \end{equation}
  is bounded (above) in probability, provided   $b$ is  bounded and   $a$  sufficiently bounded away from zero. In case of Brownian noise, it suffices to have that $a$ is strictly positive; see also \cite[Cor. 3.6]{BosGaaVer1-24}.

   \begin{proposition} \label{prop:probabove}
  		Let $Y=(Y(t))_{t\geq 0}$ be given by \eqref{eq:Y(t)} with $L=(L(t))_{t\geq 0} $ of class \textnormal{(HReg)} with jumps uniformly bounded by some $\zeta\geq 0.$   Assume      $\alpha,  \beta>0$ such  that
  	\begin{equation}
  		a(s) \geq \alpha> \lambda_N \zeta \beta
  		\quad\text{and}\quad    b(s)^{2} \leq \beta^{2}\quad \mathbb P\textnormal{-a.s.}\text{ for all $s \geq 0$,}\label{eq:assum-cor1} 
  	\end{equation} 
  with $\lambda_N\geq 0$ the rate of the  Poisson process associated to $L$.   	Then   the stochastic process in  \eqref{eq:reverse-time-sup}
  	is bounded \textnormal(above\textnormal) in probability.
  \end{proposition}

          In  {\S}\ref{sec:BMc} we  summarise and briefly discuss some of the intermediate probability estimates  in \cite{BosGaaVer1-24} for Brownian noise.  We proceed in {\S}\ref{sec:ext} by providing possible  extensions  to certain types of L\'evy processes.
  The proof of the proposition above  is to be found at the end of this section. In {\S}\ref{sec:integrands} we discuss the admissible classes of integrands in more depth, showing that one needs to be cautious using \cadlag integrators.

\subsection{The Brownian motion case}\label{sec:BMc}
Let $W=(W(t))_{t\geq 0}$ denote a standard Brownian motion on $(\Omega,\mathcal F,\mathbb F,\mathbb P)$ and consider an It\^o-process with negative drift:
 \begin{equation}
 	Y(t)=-\dint0t{a(s)}s+\dint{0}{t}{b(s)}W(s), \quad t \geq 0\label{eq:Y(t)-2},
 \end{equation}
 where $(a(t))_{t \geq 0}$ and $(b(t))_{t \geq 0}$ are measurable and adapted processes such that
 \begin{equation} \dint {0} {t}{|a(s)|}   s<\infty      \quad\text{and}\quad  \dint {0} {t}{|b(s)|^{2}}  s<\infty\quad \mathbb P\text{-a.s.},\label{A.2}  \end{equation} for all $t \geq 0$. In this setting, we observe that the stochastic integral in \eqref{eq:Y(t)-2} is a well-defined  local martingale. Assuming
 \begin{equation}\mathbb{E} \dint {0} {t}{|b(s)|^2}   s<\infty,    \quad\text{for all }t\geq 0, \end{equation}
 ensures us that the second term is a (true) square integrable martingale.  We refer to  {\S}\ref{sec:integrands} for more information regarding the conditions on the integrands.

Finding suitable estimates for the reverse time supremum
is the subject of the next two lemmas.  For  the first lemma we note  it is necessary to have a negative drift---no matter how small---to ensure an upper bound as in   \eqref{A13}. The right hand side  tends to infinity as $\alpha \searrow 0.$ An important feature of this estimate   is that
it does not depend on the length of the interval $[0, l]$ on which the supremum is
taken. In particular,  we could thus take $l\to\infty$. This result below is obtained by a time-change argument using the Dambis--Dubins--Schwarz  theorem for local martingales \cite[Thm. 4.6] {book:karatzas}.  


	 \begin{lemma}[Lem. 3.4 of \cite{BosGaaVer1-24}]\label{A12}
	 	 	Let $Y=(Y(t))_{t\geq 0}$ be a stochastic process given by \eqref{eq:Y(t)-2}.  
	 	 	 Assume that there exist positive constants $\alpha>0$ and $\beta >0$ such that
	 	 	\begin{equation}
	 	 		a(s) \geq \alpha \quad \text { and } \quad     b(s)^{2} \leq \beta^{2}\quad \mathbb P\textnormal{-a.s.}\,\, for\,\, all\,\,  s\geq 0. 
	 	 	\end{equation}
	Then for every $l \in \mathbb{N}$ and $R \geq 0$ we have
	\begin{equation}\label{A13}
	\mathbb{P}\left(\sup _{0 \leq \theta \leq l}(Y(l)-Y(\theta)) \geq R\right) \leq 4 \exp \left(-\frac{R^{2}}{64 \beta^{2}}\right)+\frac{4\exp \left(-\frac{\alpha R}{64 \beta^{2}}\right)}{1-\exp \left(-\frac{\alpha^{2}}{128 \beta^{2}}\right)}.
	\end{equation}
	 \end{lemma}

The second lemma is a Gaussian tail estimate of the supremum of a stochastic integral over an interval with a fixed length $T>0$. This result is  obtained in \cite[Lem. 3.5]{BosGaaVer1-24} again by a time-change argument. Note that the value 16 in \eqref{2exp} is not necessarily optimal. This bound is strongly related to sub-Gaussian variables, asymptotic $\mathcal O(\sqrt p)$-dependence as $p\to\infty$ of the upper universal constant\,$\sqrt[\leftroot{-2}\uproot{2}\scalebox{.7}{$p$}]{C_p}$ in the Burkholder--Davis--Gundy (BDG) inequality, and Dudley's theorem \cite{adler2009random,buldygin1980sub,kuhn2023maximal,rigollet2023high,van2020maximal,talagrand2005generic,vershynin2018high}.  For illustrative purposes, we provide an alternative proof of Lemma \ref{lemma-2} that does not exploit the time-change theorem nor the BDG-inequality (see the discussion succeeding this lemma).

  \begin{lemma}[Lem. 3.5 of \cite{BosGaaVer1-24}]
  	Let $W=(W(t))_{t\geq 0}$ be a standard Brownian motion.\label{lemma-2} Assume there exists a  $\beta>0$ such that
  	\begin{equation}
  		b(s)^2\leq \beta^2\quad \mathbb P\textnormal{-a.s.}\text{  for all $s\geq 0$}.
  	\end{equation}
  	Let $t_0\geq 0$ and $T>0$ be fixed. Then for every $R\geq 0$ we have
  	\begin{equation}
  		\mathbb P\left(\sup_{t_0\leq t\leq t_0+T}\dint{t_0}t{b(s)}W(s)\geq R\right)\leq 2\exp\left(-\frac{R^2}{16\beta^2T}\right).\label{2exp}
  	\end{equation}
  \end{lemma}
 \begin{proof}[Proof]  Let $I=[t_0,t_0+T]$ be the interval  on which we consider the martingale  $M(t)=\dint{t_0}t{b(s)}W(s),$   $t\in I,$ and define $X_T=\sup_{t_0\leq t\leq t_0+T}M(t)$. We deduce from the Borell-TIS theorem \cite[Thm. 2.1.1]{adler2009random} that $X_T-\mathbb EX_T$ is sub-Gaussian \cite[Def. 2.5.6]{vershynin2018high}, hence the moment generating function for $X_T$ as function of $\lambda$ is bounded by
 \begin{equation}
     \mathbb E\big[e^{\lambda X_T}\big]\leq e^{\lambda \mathbb E X_T+\delta\lambda ^2\sup_{t\in I}\mathbb E[M(t)^2]},
 \end{equation}
 for some $\delta>0$ \cite[Lem. 1.5]{rigollet2023high}. 
 
There exists a  $\rho>0$ such that   $x\leq \ln 2 +x^2/\rho$, for all $x\in\mathbb R$. In addition, we have $(\mathbb EX_T)^2\leq \mathbb EX_T^2\leq 4 \mathbb E[M({t_0+T})^2]$ due to Doob's maximal inequality \cite[Thm. 1.3.8]{book:karatzas}. This gives us
 \begin{equation}
     \mathbb E[e^{\lambda X_T}]\leq 2e^{c\lambda^2\sup_{t\in I}\mathbb E[M(t)^2]},
 \end{equation} where $c=4\rho^{-1}+\delta.$ We infer that $X_T$ is also sub-Gaussian; in particular, by Markov's inequality we find the so-called Chernoff bound  which satisfies
  \begin{equation}
     \mathbb P(X_T\geq R)\leq \inf_{\lambda>0}\frac{\mathbb E\big[e^{\lambda X_T}\big]}{e^{\lambda R}}\leq \inf_{\lambda>0} 2 e^{-\lambda R+c\lambda^2\sup_{t\in I}\mathbb E[M(t)^2]}=2e^{-\tfrac{R^2}{4c \sup_{t\in I}\mathbb E[M(t)^2]}},
 \end{equation}
 as the infimum is attained at $\lambda = {R}/({2c\sup_{t\in I}\mathbb E[M(t)^2]})$.

Finally, we have $\mathbb E[M(t)]^2=\mathbb E\langle M\rangle(t)=\mathbb E\int_{t_0}^tb(s)^2\mathrm ds\leq \beta^2T$, for  $t\in I$. Setting $\rho=1$ and $\delta=4$ as in \cite[Lem. 1.5]{rigollet2023high} results into \eqref{2exp} with  16 replaced by 32. Careful inspection shows that we may also consider $\rho=2.77$ and $\delta=2.552$, implying $c=4\rho^{-1}+\delta\leq 4$.
 \end{proof}

Note that the estimate above|with  constants different from 2 and 16|can also be achieved by exploiting the BDG-inequality. Recall the notation $X_T$ and $M$ above. We have the moment bounds \begin{equation}\mathbb EX_T^p\leq C_p \mathbb E[\langle M\rangle^{p/2}(t_0+T)]\leq K^pp^{p/2}(\beta^2T)^{p/2},\end{equation}for some $K>0$ and all $p\geq 2;$ we can take $C_p=10^pp^{p/2}$,  $K=10,$ according to  \cite[Thm. 5.2]{van2020maximal}; see also \cite{davis1976p}. Similar  estimates to \eqref{2exp}  readily follow from, e.g., 
 \cite[Lem. 2.2]{hamster2020stability} or \cite[Lem. 4.3]{van2020maximal}. Recall that  the standard proof of the BDG-inequality that involves It\^o's formula only results into asymptotic $\mathcal O(p)$-dependence. Indeed, from \cite[Prop. IV.4.3]{book:revuz}, we can merely deduce $C_p= 2^pp^p$, but this does not prevent us from deducing Gaussian tail estimates \cite[Lem. B.1]{van2025multidimensional}. 

 \subsection{Extensions to regulated L\'evy processes}\label{sec:ext}
In this section we derive  estimates similar to those in {\S}\ref{sec:BMc}  for Lévy-driven processes with negative drift. In particular, we consider a process
\begin{equation}
  	Y(t)=-\dint0t{a(s)}s+\dint{0}{t}{b(s)}L(s), \quad t \geq 0\label{eq:Y(t)-1},
  \end{equation}
  where $L=(L(t))_{t\geq 0}$ is  of class (HReg) and with $(a(t))_{t \geq 0}$ and $(b(t))_{t \geq 0}$ being processes in $\mathbb L[0,\infty),$ i.e., the space of adapted processes with  \caglad sample paths ($\mathbb P$-a.s.\ almost) everywhere.  Again, we refer to {\S}\ref{sec:integrands} for a discussion on the conditions imposed upon the integrands.

The following lemma shows that an estimate comparable to \eqref{A13} holds true as well in certain cases with jumps.
It seems that we cannot take $\alpha$ arbitrarily small now, as opposed to the Brownian noise case. 
This reflects the fact that all jumps need to be compensated by the negative drift. By increasing the jump height $\zeta$  one might speculate the overall control to eventually  break, yet  the proof below makes no use of  martingale properties whatsoever. We will leave it as an open problem whether  $\alpha$ needs to be bounded away from zero.


   \begin{lemma}\label{thm-1} 
  	Let $Y=(Y(t))_{t\geq 0}$ be given by \eqref{eq:Y(t)-1} with $L=(L(t))_{t\geq 0} $ of class \textnormal{(HReg)} with jumps uniformly bounded by some $\zeta\geq 0.$ Assume    there exist constants  $\alpha,\beta>0$     such  that
  	\begin{equation}
  		a(s) \geq \alpha> 2\lambda_N  \zeta \beta
  		\quad\text{and}\quad    b(s)^{2} \leq \beta^{2}\quad \mathbb P\textnormal{-a.s.}\text{ for all $s \geq 0$,}\label{eq:assum-thm1}
  	\end{equation}
   where $\lambda_N\geq 0$ is the rate of the  Poisson process $N$ associated to $L$. 
  Then there exists a $\kappa_1>0$ such that for every integer $l \in \mathbb{N}$ and $R \geq 0$ we have  
  \begin{equation}
  \begin{aligned}
  	&\mathbb{P}\left(\sup _{0 \leq \theta \leq l}(Y(l)-Y(\theta)) \geq R\right) \\&\hspace{2.5cm}\leq 4\exp \left(-\frac{R^{2}}{256\beta^2 \sigma^{2}}\right)+\frac{4\exp \left(-\frac{\alpha R}{256\beta^2 \sigma^{2}}\right)}{1-\exp \left(-\frac{\alpha^{2}}{512 \beta^2\sigma^{2}}\right)}+\exp\big(-\kappa_1 R\big).
  \end{aligned}
  \end{equation}
  \end{lemma}
  \begin{proof}
Assume $\zeta>0$ without loss of generality.  Observe we have the inequality
  	\begin{equation}\begin{aligned}
  		\mathbb{P}&\left(\sup _{0 \leq \theta \leq l}(Y(l)-Y(\theta)) \geq R\right)\\&\qquad\qquad\leq \,\mathbb P\left(\sup_{0\leq \theta\leq l}-\frac12\dint \theta l{a(s)} s+\dint \theta l{b(s)\sigma} W(s)\geq \frac R2\right)\\&\qquad\qquad\quad\quad\quad+\mathbb P\left(\sup_{0\leq \theta\leq l}-\frac12\int_\theta^la(s)\,\mathrm  d s+\int_{(\theta,l]}b(s)\,\mathrm dZ(s)\geq \frac R2\right).
  	\end{aligned}\end{equation}
	Thanks to Lemma \ref{A12}, where basically one replaces $\alpha,\beta,$ and $R$   by $\alpha/2,\beta\sigma,$ and   $R/2$, respectively,   we obtain 
  	\begin{equation}\begin{aligned}
  	  	&	\mathbb P\left(\sup_{0\leq \theta\leq l}-\frac12\int_\theta^la(s)\,\mathrm  d s+\int_\theta^lb(s)\sigma\,\mathrm dW(s)
  		\geq \frac R2\right) \\&\hspace{5cm}\leq  4\exp \left(-\frac{R^{2}}{256\beta^2 \sigma^{2}}\right)+\frac{4\exp \left(-\frac{\alpha R}{256 \beta^2\sigma^{2}}\right)}{1-\exp \left(-\frac{\alpha^{2}}{512 \beta^2\sigma^{2}}\right)}.
  		  		\end{aligned}
  	\end{equation}
  	For the remaining part, note that
   $
       \int_0^l|b(s)|\,\mathrm d|Z|(s)\leq \beta \zeta N(l)
   $ holds, hence\allowdisplaybreaks
  	\begin{align}
  	\nonumber	\mathbb P\left(\sup_{0\leq \theta\leq l}-\frac12\int_\theta^la(s)\,\mathrm d s+\int_{(\theta,l]}b(s)\,\mathrm d Z(s)\geq \frac R2\right)\hspace{-5cm}\\&\nonumber\leq \mathbb P\left(\sup_{0\leq \theta\leq l} \int_{(\theta,l]}b(s)\,\mathrm d Z(s)\geq \frac R2+\frac\alpha 2l\right)\\&\nonumber\leq \mathbb P\left(\sup_{0\leq \theta\leq l} \int_{(\theta,l]}|b(s)|\,\mathrm d |Z|(s)\geq \frac R2+\frac\alpha 2l\right)\nonumber\\
    &= \mathbb P\left( \int_{0}^l|b(s)|\,\mathrm d |Z|(s)\geq \frac R2+\frac\alpha 2l\right)\nonumber\\
  		&\leq \mathbb P\left(N(l)  \geq \frac1{2\zeta\beta}(R+\alpha l) \right)\nonumber\\
  		&\leq \exp\left(-\frac{c}{2\zeta\beta}(R+\alpha l)\right)\mathbb E \exp(cN(l)),
  	\end{align} 
  	where Markov's inequality has been applied in the ultimate line with convex function $x\mapsto\exp(cx),$ for some  unspecified constant $c>0$. Since the variable $N(l)$ is  Poisson distributed with  parameter $\lambda_N l$,     the moment generating function for $N(l)$ as function of $c$ is given by \begin{equation}\mathbb E\exp(cN(l))=\exp(\lambda_N l(\exp(c)-1)).\end{equation}
  	Subsequently, let us introduce the function
  	\begin{equation}f(\kappa_1)=\frac{\lambda_N}{\kappa_1}(\exp(2\kappa_1\zeta\beta)-1),\quad \kappa_1>0.\label{eq:function}\end{equation}
  	An elementary computation shows that $f$ is increasing and that we have
  	\begin{equation}
  			\lim_{\kappa_1 \searrow 0}\frac{\lambda_N}{\kappa_1}(\exp(2\kappa_1\zeta\beta)-1)=2\lambda_N \zeta \beta.
  	\end{equation}
  	So, there exists a $\kappa_1>0$ with $\alpha\geq f(\kappa_1)>2\lambda_N\zeta\beta.$ Now take such a $\kappa_1$,   set $c=2\kappa_1\zeta\beta$, and 	 let us introduce the constant $\mu=\lambda_N(\exp(c)-1)$. This yields
   	\begin{equation}\begin{aligned}
  		\mathbb P\left(\sup_{0\leq \theta\leq l}-\frac12\int_\theta^la(s)\,\mathrm d s+\int_{(\theta,l]}b(s)\,\mathrm d Z(s)\geq \frac R2\right) &\leq \exp\big(-\kappa_1(R+\alpha l)\big) \exp (\mu l )\\
  		&=\exp (-\kappa_1 R)\exp \big(l(\mu-\kappa_1\alpha)\big)\\
  		&\leq \exp (-\kappa_1 R ),
  	\end{aligned}\end{equation}
  	with the last inequality being immediate due to the fact that $\mu=\kappa_1f(\kappa_1)\leq \kappa_1 \alpha$ holds.  
   \end{proof}
  

  \begin{remark}\label{remark:sharp} The restriction $\alpha>2\lambda_N \zeta\beta$ in  equation \eqref{eq:assum-thm1} is however stronger than necessary. Let us reconsider the split up
   \begin{align*}
  		\mathbb{P}&\left(\sup _{0 \leq \theta \leq l}(Y(l)-Y(\theta)) \geq R\right)\\&\qquad\qquad\leq \,\mathbb P\left(\sup_{0\leq \theta\leq l}-(1-q)\dint \theta l{a(s)} s+\dint \theta l{b(s)\sigma} W(s)\geq (1-p) R\right)\yesnumber\\&\qquad\qquad\quad\quad\quad+\mathbb P\left(\sup_{0\leq \theta\leq l}-q\int_\theta^la(s)\,\mathrm  d s+\int_{(\theta,l]}b(s)\,\mathrm dZ(s)\geq pR\right),
  	\end{align*}
  	for any reals $0<p,q<1$, where $p=q=\frac12$ is set in the proof of Lemma \ref{thm-1}.  Following the exact same steps, the function in \eqref{eq:function} now reads
  	\begin{equation}
  		f(\kappa_1)=\frac {\lambda_N} {\kappa_1} \frac{p}{q}\big(\exp(p^{-1}\kappa_1\zeta\beta)-1\big),\quad \kappa_1>0,
  	\end{equation}
  	for any $p$ and $q$ fixed.
   Letting $\kappa_1\searrow 0$ yields that $f(\kappa_1)$ tends to $q^{-1}\lambda _N\zeta\beta$. It seems the $p$ cancels out and  essentially plays no  role. We conclude that\begin{equation}\label{eq:restriction}\alpha>\lambda_N \zeta \beta\end{equation}is a sufficient condition  after finding appropriate values for $\kappa_1$ as $q\nearrow 1$.
  \end{remark}

As a little side note, observe that  we are not able to prove Lemma \ref{thm-1} by means of a time-change argument for $L$ in (HReg). Indeed, in  \cite[Thm. 2]{kallsen2002time} it is shown that a time-change theorem exists
 for (symmetric) $\alpha$-stable L\'evy processes, and that it cannot be extended to any other
class of L\'evy processes \cite[Thm. 4]{kallsen2002time}. We will not need it here, but one could revisit the proof of Lemma \ref{A12} and try to show a similar result  for L\'evy processes that are $\alpha$-stable.

Lemma \ref{lemma-2} extends to regulated Lévy processes without additional constraints. We would like to point out that  an exponential tail estimate like \eqref{2exp} can be achieved by means of the BDG-inequality as written  
in \cite[Thm. 4.19]{kuhn2023maximal}. Now having a maximal jump height $\zeta$ turns out useful, because this implies  $\sup_{t\geq 0}|\Delta L(t)|^{p}\leq \zeta^p$, where $\Delta L(t)=L(t)-L(t-)$ denotes the jumps of $L$. The parameter $\zeta$ would turn up in the denominator of the exponential tail estimate. Conform to  the continuous case, we infer the bound  from the elementary inequality in \cite[Lem. B.1]{van2025multidimensional}; for general \cadlag martingales, only asymptotic $\mathcal O(p)$-dependence holds for the upper universal constant\,$\sqrt[\leftroot{-2}\uproot{2}\scalebox{.7}{$p$}]{C_p}$, as can be seen in, e.g., \cite[Thm. 1.9.7]{liptser2012theory} and \cite{pathwiseBDG2020}.

We  will now prove an exponential tail estimate without invoking the BDG-inequality.




  \begin{lemma}\label{thm-2}
  	Let $Y=(Y(t))_{t\geq 0}$ be given by \eqref{eq:Y(t)-1} with $L=(L(t))_{t\geq 0} $ of class \textnormal{(HReg)} with jumps uniformly bounded by some $\zeta\geq 0.$  Assume there exists a constant $\beta>0$ such that
  	\begin{equation}
  		b(s)^2\leq \beta^2\quad \mathbb P\textnormal{-a.s.}\text{  for all $s\geq 0$}.
  	\end{equation}
  	Let $t_0\geq 0$ and $T>0$ be fixed. Also, take $\kappa_2>0$ to be an arbitrary value. Then there exists a constant $R_0\geq 0$ such that for every $R\geq 0$ we have
  		\begin{equation}\begin{aligned}
  		\mathbb P&\left(\sup_{t_0\leq t\leq t_0+T}\dint {(t_0,t]}{}{b(s)}L(s)\geq R\right)\\&\qquad\qquad\qquad\qquad\leq 2\exp\left(-\frac{R^2}{64\beta^2\sigma^2T}\right)+ C\exp\big(-\kappa_2 R\big)+\mathbf 1_{\{R< R_0\}},\label{eq:upperbnd}
  	\end{aligned}\end{equation}
  	where $C=C(\lambda_N,T,\zeta,\beta,\kappa_2)\geq 1$ is a constant not depending on $R$ and $t_0$, and $\lambda_N\geq 0$  the rate of the   Poisson process $N$ associated to $L$. If $L$ is of class \textnormal{(HRegM)}, then we may set $R_0=0.$
  \end{lemma}
  \begin{proof}
  As before, consider the inequality\begin{equation}\begin{aligned}
  		\mathbb P\left(\sup_{t_0\leq t\leq t_0+T}\dint {(t_0,t]}{}{b(s)}L(s)\geq R\right)&\leq \mathbb P\left(\sup_{t_0\leq t\leq t_0+T}\int_{t_0}^tb(s)\sigma\, \mathrm dW(s)\geq \frac R2\right) \\&\quad\quad\quad\quad + \mathbb P\left(\sup_{t_0\leq t\leq t_0+T}\dint {(t_0,t]}{}{b(s)} Z(s)\geq \frac R2\right).
  	\end{aligned}\end{equation}
  	Appealing to Lemma \ref{lemma-2} with  $R$ and $\beta$ replaced by   $R/2$ and $\beta\sigma$, respectively,  gives us
  	\begin{equation}
  		\mathbb P\left(\sup_{t_0\leq t\leq t_0+T}\int_{t_0}^tb(s)\sigma \,\mathrm dW(s)\geq \frac R2\right)\leq 2\exp\left(-\frac{R^2}{64\beta^2\sigma^2T}\right).
  	\end{equation}
  	Let $M=(M(t))_{t\geq 0}$ be the martingale part of $Z=(Z(t))_{t\geq 0}$, i.e.,\begin{equation}M(t)=Z(t)-\lambda_N\mathbb EZ_1t,\quad t\geq 0.\end{equation} We obtain the inequality
   \begin{equation}\begin{aligned}
  		\mathbb P\left(\sup_{t_0\leq t\leq t_0+T}\dint {(t_0,t]}{}{b(s)}Z(s)\geq \frac R2\right)&\leq \mathbb P\left(\sup_{t_0\leq t\leq t_0+T}\int_{t_0}^tb(s)\, \mathrm dM(s)\geq \frac R4\right) +\mathbf 1_{\{R<R_0\}},
  	\end{aligned}\end{equation}
   for some $R_0\geq 0$ sufficiently large. Note that this splitting is not needed when $\mathbb EZ_1=0$, thus when $Z$ is a martingale. In that case, this allow us to disregard the indicator function, or differently put, we then may set $R_0=0.$ 
   
Subsequently, introduce the process $X=(X(t))_{t\geq 0}$, defined by
  	\begin{equation}
  		X(t)=\begin{cases}\label{eq:X}
  			\displaystyle \dint{(t_0,t]}{}{b(s)}M(s),&t>t_0,\\
  			0,&0\leq t\leq t_0,
  		\end{cases} 
  	\end{equation}
  and   observe   that $X$ is a square integrable martingale. Indeed, we have
  \begin{equation}
  \begin{aligned}
  	\mathbb E[X(t)|\mathcal F_s]&=\mathbb E\left[\dint0t{b(s)}M(s)\middle|\mathcal F_s\right]-\mathbb E\left[\dint0{t_0}{b(s)}M(s)\middle|\mathcal F_s\right]\\&= \dint0{ s}{b(s)}M(s)-\dint0{t_0\wedge s}{b(s)}M(s)
  	\\&= X(s),
  \end{aligned}\end{equation}
for  $t\geq s\geq 0$. Doob's supremal inequality (combine \cite[Prop. 1.3.6]{book:karatzas} and \cite[Thm. 1.3.8]{book:karatzas}) yields
  \begin{equation}
  	\mathbb	P\left(\sup_{t_0\leq t\leq t_0+T}X(t)\geq \frac R4\right)   \leq \exp(-\kappa_2 R)\mathbb E\exp\left(4\kappa_2 \left|\dint{t_0}{t_0+T}{b(s)}M(s)\right|\right),
  \end{equation}
where we have taken the   convex function $x\mapsto \exp(4\kappa_2 x)$. Next, observe that
\begin{equation}
	\dint{t_0}{t_0+T}{b(s)}M(s)=\dint{0}{T}{b(s+t_0)}\bar M^{t_0}(s)
\end{equation}
  holds, with $\bar M^{t_0}=(\bar M^{t_0}(t))_{t\geq 0}$ being the process defined by $\bar M^{t_0}(t)=M({t_0+t})-M({t_0}),$  $t
  \geq 0.$ In particular, we have  that $M$ and $\bar M^{t_0}$ coincide in law due to the strong Markov property \cite[Thm. I.32]{book:protter}.
  Let $N^{t_0}=(N^{t_0}(t))_{t\geq 0}$ be the  Poisson process associated to $\bar M^{t_0}$. 
This gives us
  	\begin{equation}
  	\left|\dint0T{b(u+t_0)}\bar M^{t_0}(u)\right|\leq  \dint0T{|b(u+t_0)|} |\bar M^{t_0}|(u)\leq \zeta \beta N^{t_0}(T)+\lambda_N\zeta\beta T,
  \end{equation}
holding $\mathbb P$-almost surely, hence
  	\begin{equation}\begin{aligned}
  		\mathbb	P\left(\sup_{t_0\leq t\leq t_0+T}\int_{t_0}^tb(s)\,\mathrm d M(s)\geq \frac R4\right) 
  		&\leq \exp(-\kappa_2 R) \exp\big(4\kappa_2\lambda_N\zeta\beta T\big)\mathbb E\exp\big(4\kappa_2 \zeta \beta N^{t_0}(T)\big)\\
    &=C(\lambda,T,\zeta,\beta,\kappa_2)\exp(-\kappa_2R),
  	\end{aligned}\end{equation}
where  the constant $C=C(\lambda,T,\zeta,\beta,\kappa_2)=\exp(4\kappa_2\lambda_N\zeta\beta T)\exp\left(\lambda_N T(\exp({4\kappa_2 \zeta\beta})-1)\right)\geq 1$  indeed does  not  depend on the parameters   $R$ and $t_0$. 
  \end{proof}

  \begin{proof}[Proof of Proposition \ref{prop:probabove}] 
 Let $R\geq 0$ and   $t \geq 0 $ be arbitrary. Consider $l$ to be the greatest integer below $t,$ i.e., we want $l$ to satisfy the condition $\lfloor t\rfloor -1<l\leq \lfloor t \rfloor$. Completely analogous to the proof of \cite[Cor. 3.6]{BosGaaVer1-24} in the case of Brownian noise, we find
 \begin{equation}
  	\begin{aligned}
  		\mathbb{P}\left(\sup _{\theta \in[0, t]} \big(Y(t)-Y(\theta) \big)\geq R\right) &
  		\leq \mathbb{P}\left(\sup _{\theta \in[0, l]} -\int_{\theta}^{l} a(s) \,\mathrm d s+\int_{(\theta,l]}^{ } b(s)\,\mathrm d L(s)  \geq \frac R  3\right) \\
  		&\quad\quad\quad+\mathbb{P}\left( \sup _{\theta \in[t-1, t]}\dint{(t-1,\theta]}{} {b(s)}	L(s)\geq \frac R  3\right) \\
  		&\quad\quad\quad\quad\quad\quad+\mathbb{P}\left( \sup _{\theta \in[t-1, t]}\dint{(t-1,\theta]}{} {\big(-b(s)\big)}	L(s) \geq \frac R  3\right).
  	\end{aligned}
   \end{equation}
  	The latter is below a given $\varepsilon>0$, for any   $t \geq 0,$ provided   $R$ is   big enough. 
Thus, the result we want to prove is a  direct consequence of Lemma \ref{thm-1}, together with Remark \ref{remark:sharp},    and Lemma \ref{thm-2}.
  \end{proof}

\subsection{Admissible classes of integrands}\label{sec:integrands}

On a historical note, 
  lots of literature, e.g., \cite{book:kallenberg,book:karatzas,book:revuz}, rather work with progressively measurable integrands because progressive processes satisfy nice properties. For this class of integrands we have that the Lebesgue--Stieltjes integral in \eqref{eq:Y(t)} is again adapted. 
 Nonetheless,  if a stochastic 
process $(c(t))_{t\geq 0}$ is measurable and adapted, then  a progressively measurable modification, say $(\hat c(t))_{t\geq 0}$, exists \cite[p. 68]{book:meyer}. This then implies for measurable adapted processes  that the first integral in \eqref{eq:Y(t)} is indistinguishable from an adapted process, hence  adapted itself since the filtered probability space satisfies the usual conditions \cite[Lem. 3.11]{book:chung}.

Furthermore, as   briefly pointed out in \cite[Sec. 1.5]{book:mao}, a progressively measurable process $(\hat c(t))_{t\geq 0}$ enables us to construct a predictable process $(\bar c(t))_{t\geq 0}$, namely
\begin{equation}
    \bar c(t)=\limsup_{h\downarrow 0}\frac{1}{h}\int_{t-h}^t\hat c(s)\,\mathrm ds.
\end{equation}
This is then a predictable modification of $(\hat c(t))_{t\geq 0}$ assuming the integrand is in $L^1,$ i.e., the first integrability condition in \eqref{A.2} is satisfied for the process $(\hat c(t))_{t\geq 0}$, which follows from Lebesgue's differentiation theorem \cite[Thm. 5.6.2]{bogachev2007measure}. 
In particular, this implies  that $ (\mathrm ds\times\mathbb P)((t,\omega)\in[0,\infty)\times\Omega:c(t,\omega)\neq \bar c(t,\omega))=0$. Consequently, we may replace the measurable and adapted processes $(a(t))_{t \geq 0}$ and $(b(t))_{t \geq 0}$ in {\S}\ref{sec:BMc} by  either progressively or predictably measurable stochastic processes without loss of generality, as the corresponding It\^o-processes with negative drift are indistinguishable from one another. 

The class of admissible integrands in {\S}\ref{sec:ext} can actually be expanded. The results in this section  hold true for $(a(t))_{t \geq 0}$ and $(b(t))_{t \geq 0}$  measurable and adapted processes such that
\begin{equation}  \dint {0} {t}{|a(s)|}   s<\infty   \quad \mathbb P\text{-a.s.} \label{eq:sufficient}  \quad\text{and}\quad \mathbb{E} \dint {0} {t}{|b(s)|^{2}}  s<\infty, \quad t\geq 0.  \end{equation}
We will not need it in this paper, but it is worth mentioning here. We can assume without loss that  $L$  is of class (HRegM) for the explanation below.
Following \cite[Ch. IV]{book:protter}, the integrand $(b(t))_{t \geq 0}$ can be  a  predictable process because of the following identity:
 \begin{equation}
      \mathbb E\int_0^t|b(s)|^2\mathrm d[L](s)=\mathbb E\int_0^t|b(s)|^2\mathrm d\langle L\rangle (s)=\lambda \mathbb E\int_0^t|b(s)|^2\mathrm ds.\label{eq:calcul}
  \end{equation}
Recall, the first term in \eqref{eq:Y(t)} can  be interpreted as a stochastic integral with respect to a semimartingale if $(a(t))_{t \geq 0}$ is in $\mathbb L[0,\infty)$. In case $(a(t))_{t \geq 0}$ is assumed to be predictable,   a stronger integrability assumption is required and  it then still coincides with the Lebesgue--Stieltjes interpretation.
   Even more is true. Since (HRegM)\,$\subseteq$\,(HDol)\footnote{See \cite[App. A]{artikel1-general} for the definition of (HDol).}, one can follow \cite{book:chung,unpublished:timo} to have the process
   \begin{equation}
  	\left(\dint{0}{t}{b(s)}L(s)\right)_{t\geq 0} \label{eq:stoc_int}
  \end{equation}
be well-defined  as a square integrable martingale for measurable and adapted processes $(b(t))_{t\geq 0}$.
As before, and according to \cite[{p. 66}]{book:chung} and \cite[p. 190]{unpublished:timo}, we have that for any measurable and adapted processes $(b(t))_{t\geq 0}$ there is a predictable process  $(\bar b(s))_{t\geq 0}$ such that   $ (\mathrm ds\times\mathbb P)((t,\omega)\in[0,\infty)\times\Omega:b(t,\omega)\neq \bar b(t,\omega))=0$. Finally,  $\mu_M\ll \mathrm ds\times \mathbb P$ yields $ \mu_M((t,\omega)\in[0,\infty)\times\Omega:b(t,\omega)\neq \bar b(t,\omega))=0$, which allows for a martingale extension of the stochastic integral.  

    An important warning! For adapted and measurable integrands, we note that a stochastic integral 
    for which the integrator is also of finite variation  may no longer coincide with the Lebesgue--Stieltjes integral. For example, consider the martingale $\smash[b]{M(t)=N(t)-\lambda_N t}$, $t\geq 0,$ where 
    $\smash[b]{(N(t))_{t\geq 0}}$ is a Poisson process with intensity $\lambda_N.$ Observe that $\smash[b]{\int_0^tN(s)\,\mathrm dM(s)}$, interpreted  as Lebesgue--Stieltjes integral, equals $ \int_0^tN(s-)\,\mathrm dM(s)+N(t), $
      while it  coincides with   $\int_0^tN(s-)\,\mathrm dM(s)$ when it is being interpreted as a stochastic integral. To solve this remedy, one should either restrict to the class of predictable integrands or be very  clear about how   the second term in \eqref{eq:Y(t)}, i.e., stochastic process \eqref{eq:stoc_int}, should be interpreted. We refer to \cite{Bosch24} for a more elaborate discussion. 


\section{Delay equations with stochastic negative feedback}\label{Sec6}
This section is dedicated to proving Theorem \ref{main-thrm}. In particular,  the results  in this section  allow for a natural extension of the latter to L\'evy processes $L=(L(t))_{t\geq 0}$ of type (HReg), as discussed in the introduction. Our approach involves analysing the transformed equation
  \begin{equation}
  \mathrm d Y(t)=\left[-\gamma(t)+r(t)e^{-Y(t)}f\big(e^{Y(t-1)}\big)\right] \mathrm d t+a\left(Y_{t},t\right) \mathrm d t+b\left(Y_{t-},t-\right) \mathrm d L(t),\quad Y_0=\Phi,\label{eq:DESNF}
  \end{equation}
  with $a,b:D[-1,0]\times \mathbb R\to\mathbb R$ being time-proper locally Lipschitz functionals, $\Phi\in \mathbb D[-1,0]$, and the reproduction rate $r:[0,\infty)\to[0,\infty)$ and  the mortality rate $\gamma:[0,\infty)\to[0,\infty)$ are considered \cadlag and positive. If  $L$ is  a Brownian motion, it suffices to assume measurablity of the rates $r(t)$ and $\gamma(t)$, and the rates  will be   set to constant when we search for (non-trivial)  invariant measures. 
Recall, we have seen in {\S}\ref{sec:new:3} that for existence of an invariant measure, hence that of a stationary distribution and  solution, it suffices to show global existence of all solutions as well as boundedness in probability of at least one solution, i.e., for some initial process $\Phi$. 

The following  method is in line with the proof strategy of  the companion paper on the stochastic Wright's equation \cite{BosGaaVer1-24}. The  idea is to  keep  track of the  trajectories of stochastic solutions to \eqref{eq:DESNF} in a pathwise manner, which is the subject of {\S}\ref{sec:pathwise}. Next, in {\S}\ref{sec:global} we show that the just obtained  solution estimates lead to  global existence and that actually all solutions are bounded above in probability; the latter follows from utilising  the  estimates for It\^o- and Lévy-driven processes  with negative drift in the previous section.
The value  of $f(0)=\lim_{x\searrow 0}f(x)$  plays an intricate role with regard to having solutions  bounded below in probability. (Note that the value of $f(x)$, for $x< 0$, has no effect on the problem at all.) In this paper, we focus on the case $f(0)>0$ and show that all solutions to \eqref{eq:DESNF} are also bounded below in probability. Various phenomena can occur in the case  $f(0)=0$ and requires a more in-depth investigation.

 \subsection{Solution estimates of integral delay equations with negative feedback and deterministic forcing}\label{sec:pathwise}
In this section, we provide an upper  bound on the solution for any function $f$ that is bounded from above, together with a lower  bound on the solution in  case  $f(0)>0$ is satisfied. Before we proceed, let us provide a deterministic version of what is  meant by being regulated,  in line with  class (HReg). 
 

  \begin{definition}
  	A \cadlag function $v:E\subset \R \to\R$ is called  \texttt{regulated} whenever the amount of discontinuities on compact intervals is finite\footnote{This  finite intensity property   in the definition of   a regulated function is, as a matter of fact, not necessary for any result in this section; all proofs simply do not exploit this property. The condition is  therefore superfluous, but we keep it as it is   in line with our   ``regulated Lévy process'' terminology.} and   the corresponding jumps are uniformly bounded by some constant $\zeta\geq 0$, i.e., $|v(t-)-v(t)|\leq \zeta$ for all $t\in E$. 
  \end{definition}


\begin{lemma}\label{prop:MG}
Suppose  $f:\mathbb R\to \mathbb R$ is  continuous, non-negative on $(0,\infty)$, and bounded from above by some constant $M> 0$.
In addition, let $ \tilde \gamma:=\inf_{t\geq 0}\gamma(t)>0,$ $     \tilde r:=\sup_{t\geq 0} r(t)<\infty, $ and $ t_{0} \geq 0$. Assume that   $z:[t_0-1, \infty) \rightarrow \mathbb{R}$ is \cadlag and satisfies the integral equation
\begin{equation}\label{eq:integraleq}
z(t)=z (t_{0} )+\int_{t_{0}}^{t}\left[-\gamma(\theta)+ r(\theta)e^{-z(\theta)}f\big(e^{z(\theta-1)}\big) \right]\mathrm  d \theta  +v(t)-v (t_{0} ),\quad t\geq t_0,
\end{equation}
where $v:[t_0, \infty) \rightarrow \mathbb{R}$ is   regulated  with maximal jump height $\zeta\geq 0$. Let $R\geq 0$ be such that
\begin{equation}
z(t_0)<R.
\end{equation}
Then for every $t \geq t_0$ there exists  an   $a^{t} \in[t_0, t]$ such that  the solution $z(t)$  satisfies
\begin{equation}
z(t) \leq \max \left\{R,R+\zeta
-\alpha (t-a^{t} )
+v(t)-v (a^{t} )\right\},\label{eq:final}
\end{equation}
where  $\alpha =\tilde \gamma-\tilde r Me^{-R}.$
\end{lemma}
\begin{proof}

	Fix $t \in[t_0, \infty) $ and note that either $z(t) < R$ or $z(t)\geq R $ holds. In case of the latter, define
\begin{equation}\label{eq:randomv}
a^{t} :=\sup \{s \in[t_0, t]: z(s)<R\}.
\end{equation}
As we have $z(t_0)<R$  by assumption, we deduce that the supremum  is taken over a non-empty set, hence the supremum is finite and $a^t\in[t_0,t]$.  
  Further, we have $z(a^t)\geq R$, $z(a^t-)\leq R$, and
  \begin{equation}
  	z(a^t)\leq z(a^t-)+\zeta \leq R+\zeta.
  \end{equation}
   This is  because the function $z(t)$ is  regulated too,  with jumps uniformly bounded by again $\zeta$, since only  $v(t)$ causes   discontinuities in $z(t)$. Note that this  is immediate from the fact that  the integral part in \eqref{eq:integraleq}  is continuous in $t$.
   
  Moreover, we have $z(\theta) \geq R$   for all $\theta \in\left[a^{t}, t\right].$ This yields
\begin{equation}
\begin{aligned}
	z(t)&=  z (a^{t} )+\int_{a^{t}}^{t}\left[-\gamma(\theta)+r(\theta)e^{-z(\theta)}  f\big(e^{z(\theta-1)}\big) \right]\mathrm  d   \theta 
	 +v(t)-v (a^{t} ) \\
	&\leq  R+\zeta+\int_{a^{t}}^{t} (-\tilde \gamma +\tilde r M e^{-R})\,  \mathrm  d \theta +v(t)-v (a^{t} ) \\
	&=  R+\zeta- (\tilde \gamma -\tilde r M  e^{-R}) (t-a^{t} )+v(t)-v (a^{t} ).
\end{aligned}
\end{equation}
In conclusion, the  solution   indeed satisfies \eqref{eq:final} as the time $t$ was chosen arbitrarily. 
\end{proof}

The following lemma will be used to obtain a lower bound on the solution to \eqref{eq:integraleq}. This lemma generalises \cite[Lem. 4.1]{BosGaaVer1-24}, as it allows for jump occurrences to happen, and may also be invoked to study the stochastic Wright's equation driven by a regulated L\'evy process.


  \begin{lemma}
  	Let $\beta\geq0$, $\zeta\geq 0$, \label{thm:q=0} and suppose $F:(0,\infty)\times (0,\infty)\to \R$   is continuous,
  	satisfying the assumption that there exist  constants $\delta>0$
  	and $C_F\geq 0$ such that
  	\begin{equation} 
  	\label{eq:q=0-condit1}	F(x,y)\geq  \beta\quad\text{for all } x,y\in(0,\delta),
  	\end{equation} 
  	and
  	\begin{equation}
  		F(x,y)\geq -C_F\quad\text{for all }x,y\in(0,\infty).
  	\end{equation}
  	Let $t_0\geq 0$.  Assume  that $z:[t_0-1,\infty)\to\R$ is \cadlag and satisfies 
  	the integral equation
  	\begin{equation}
  		z(t)=z(t_0)+\int_{t_0}^t F\big(e^{z(\theta-1)},e^{z(\theta)}\big)\,\mathrm d\theta-\beta(t-t_0)+v(t)-v(t_0),\quad t\geq t_0,
  	\end{equation}
where $v:[t_0, \infty) \rightarrow \mathbb{R}$ is   regulated  with maximal jump height $\zeta\geq 0$. Let $R\geq 0$ be such that
  	\begin{equation}
  		e^{-R}<\delta\quad\text{and}\quad z(t_0)>-R.
  	\end{equation}
  	Then for every $t \geq t_0$ there exists a time $a^{t} \in[t_0, t]$ such that  the solution $z(t)$  satisfies
  	\begin{equation}
  		\label{eq:lowbound1}	z(t)\geq\min\{-R, -R-C_F-\beta-\zeta+v(t)-v(a^t)\}.
  	\end{equation}
  \end{lemma}
  
  \begin{proof} As before, fix $t\in [t_0,\infty)$, observe that    either  $z(t)>-R$ or $z(t)\leq -R$ holds, and from now one we shall assume the latter scenario.  Again,   we  define
  	\begin{equation}
  		a^t:=\sup\{s\in[t_0,t]:z(s)>-R\},
  	\end{equation}
  note that $a^t\in[t_0,t]$  because $z(t_0)>-R$, and we have
  \begin{equation}
  	z(a^t)\geq z(a^t-)-\zeta\geq -R-\zeta,\label{eq:regulated!}
   \end{equation}
  since the solution $z(t)$ is regulated too with jumps uniformly bounded by the same constant $\zeta.$ 
   If we have $t-a^t<1$, then we can simply conclude
  	\begin{equation}\begin{aligned}
  		  		z(t)&=z(a^t)+\dint{a^t}t{ F\big(e^{z(\theta-1)},e^{z(\theta)}\big)}\theta-\beta(t-a^t)+v(t)-v(a^t)\\
  		  		&\geq -R-\zeta-\dint{a^t}{t}{C_F}\theta-\beta(t-a^t)+v(t)-v(a^t)\\
  		  		&\geq -R-C_F-\beta-\zeta+v(t)-v(a^t).
  		  	\end{aligned}\end{equation}
  	  	On the other hand, if $t-a^t\geq 1$ is the case, then we need to exploit the fact that $z(\eta)\leq -R$ holds  for all $\eta\in[a^t,t]$, and
  	  	so we have
  	  	\begin{equation}
  	  		e^{z(\eta)}\leq e^{-R}<\delta,\quad\text{for all }\eta\in [a^t,t].
  	  	\end{equation} 
  	This gives us
\begin{equation}\begin{aligned}
  		z(t)&=z(a^t)+\dint{a^t}t{ F\big(e^{z(\theta-1)},e^{z(\theta)}\big)}\theta-\beta(t-a^t)+v(t)-v(a^t)\\
  		&=z(a^t)+\dint{a^t}{a^t+1}{ F\big(e^{z(\theta-1)},e^{z(\theta)}\big)}\theta\\
  		&\hspace{3cm}+\dint{a^t+1}t{ F\big(e^{z(\theta-1)},e^{z(\theta)}\big)}\theta-\beta(t-a^t)+v(t)-v(a^t)\\
  		&\geq	
  		z(a^t)-\dint{a^t}{a^t+1}{C_F}\theta+\dint{a^t+1}{t}{\beta}\theta-\beta(t-a^t)+v(t)-v(a^t)\\
  		&= z(a^t)-C_F+\beta(t-a^t-1)-\beta(t-a^t)+v(t)-v(a^t)\\
  		&= z(a^t)-C_F-\beta+v(t)-v(a^t)\\
  		&\geq -R-C_F-\beta-\zeta+v(t)-v(a^{t}), 
  	\end{aligned}\end{equation}
  	where the last inequality follows from equation  \eqref{eq:regulated!}.
  	This completes the proof.
  \end{proof}

\begin{corollary}\label{cor:q=0}
Suppose that $f:\R\to\R$ is  continuous, non-negative on $(0,\infty)$, and       $f(0)>0.$ 	Let $\gamma, r> 0$ and  $ t_0\geq 0$. Assume that $z:[t_0-1,\infty)\to \R$ is \cadlag and satisfies the integral equation
	\begin{equation}\label{eq:integral1}
		z(t)=z(t_0)-\gamma(t-t_0)+\int_{t_0}^t re^{-z(\theta)}f\big(e^{z(\theta-1)}\big)\,\mathrm d\theta+v(t)-v(t_0),\quad t\geq t_0,
	\end{equation}
where $v:[t_0, \infty) \rightarrow \mathbb{R}$ is   regulated  with maximal jump height $\zeta\geq 0$.
	Let $R\geq 0$ be sufficiently large. Then, if $z(t_0)>-R$, we have 
  for every $t \geq t_0$ there exists a time $a^{t} \in[t_0, t]$ such that 
\begin{equation}
 	z(t)\geq \min\{-R, -R-\gamma-\zeta+v(t)-v(a^t)\}.
\end{equation}
\end{corollary}
This corollary is an immediate consequence of Lemma \ref{thm:q=0} by taking
	$F(x,y)=f(x)y^{-1}.$
Indeed, for sufficiently small $\varepsilon>0$, we are able to find a $\delta>0$ small enough such that $f(x)\geq \varepsilon$ holds for $x\in(0,\delta)$ and with $y^{-1}\geq \delta^{-1}\geq \gamma/\varepsilon$ for $y\in (0,\delta).$ Combining yields  $F(x,y)\geq \gamma$.
We leave it as an exercise to the reader to state a similar result with $\gamma=\gamma(t)$ and $r=r(t)$ time-dependent.

In the case of $f(0)= 0,$ the problem of finding a lower bound becomes much more delicate. We believe that by writing $f(x)=xg(x)$ with $g(0)>0,$ we can introduce $F(x,y)=xy^{-1}g(x)$ in which we might be able to exploit the full potential of the $y^{-1}$ term.
Currently, generalising Lemma \ref{thm:q=0} and finding effective lower bounds suitable for the approach in  {\S}\ref{sec:global}, hence finding estimates which do not grow logarithmically in time when $v$ is a stochastic process, is still open.

\subsection{Global existence, boundedness in probability, and stationarity}\label{sec:global}


In this final section, we first show that almost all solutions to  \eqref{eq:DESNF}
persist globally and are bounded above in probability. We aim to showcase the potential of the newly developed tools in {\S}\ref{Sec5} and {\S}\ref{sec:pathwise}, as  boundedness  in probability from above  can also be deduced  from Proposition \ref{prop:alternatief} (assuming that the solutions are global). 
Although the proof of Proposition \ref{prop:alternatief}
 is rather concise and requires minimal machinery, there is no need in the method  below for the noise part to be a (local) martingale. This method additionally enables us to  prove that all solutions are indeed global.

  \begin{proposition}\label{thm:boundabove}
 	Consider the non-autonomous stochastic delay differential equation
  \begin{equation}
  \mathrm d Y(t)=\left[-\gamma(t)+r(t)e^{-Y(t)}f\big(e^{Y(t-1)}\big)\right] \mathrm d t+a\left(Y_{t},t\right) \mathrm d t+b\left(Y_{t-},t-\right) \mathrm d L(t),\label{eq:MGeq1}
  \end{equation}
where $L=(L(t))_{t\geq 0}$ is of class \textnormal{(HReg)} with jumps uniformly bounded by  $\zeta\geq 0$. Let   $\lambda_N\geq 0$ denote the rate of the   Poisson process  associated to $L$.
Suppose $f:\mathbb R\to \mathbb R$ is  locally Lipschitz continuous, non-negative on $(0,\infty)$, and bounded from above. Assume $\tilde \gamma:=\inf_{t\geq 0}\gamma(t)>0$ and $\sup_{t\geq 0}r(t)<\infty$.
If there exists non-negative constants $\alpha_{\rm max}\geq 0$ and $ \beta \geq 0$ such that
\begin{equation}
 a(\varphi,\,\cdot\,)\leq  \alpha_{\rm max}  \quad \text { and } \quad b(\varphi,\,\cdot\,)^{2} \leq \beta^{2}, \quad \text { for all } \varphi \in D[-1,0],
  \end{equation}
  and if
  \begin{equation}\label{eq:byproduct}
  	\tilde \gamma >\alpha_{\rm max}+\lambda_N \zeta\beta,
  \end{equation}
  then for any $\Phi \in \mathbb D[-1,0]$ we have that the solution to equation \eqref{eq:MGeq1}  with $Y_0=\Phi$ persists globally. Furthermore, any solution is bounded above in probability. 
\end{proposition} 
\begin{proof}
	Let us denote by  $(Z,T_\infty)$   the maximal local solution to  equation \eqref{eq:MGeq1}, subject to any  initial process $\Psi\in\mathbb D[-1,0]$ with $\Psi(0)$   bounded from above by some  constant $R'>0.$	
  Proposition 3.4 of \cite{artikel1-general} tells us  that solutions can only tend to $+\infty$ in finite time.

We proceed by introducing the  process $(V(t))_{-1\leq t<\infty}$, defined for $0\leq t<T_\infty$ by
  \begin{equation}
  V(t)=\int_{0}^{t} a\left(Z_s,s\right) \mathrm d s+\int_{0}^{t} b\left(Z_{s-},s-\right) \mathrm dL(s),\label{eq:thisform}
  \end{equation}
  while $V(t)=0$ for $t\in [-1,0]$ and $t\in [T_\infty,\infty).$
Consider a sequence of stopping times $(T_k)_{k\geq 1}$   as in \cite[Def. 2.1]{artikel1-general}. Then, for every  $t \geq t_{0} \geq 0$, we have
  \begin{equation}
  	Z(t\wedge T_k)=Z(t_{0})+  \int_{t_{0}}^{t\wedge T_k}\left[-\gamma(s)+ r(s)  e^{-Z(s)}f\big(e^{Z(s-1)}\big)\right]\mathrm d s \\
  	  +V(t\wedge T_k)-V(t_{0}\wedge T_k),
  \end{equation}
for all integers $k\geq 1$.
  Choose $R>0$  large enough such that it satisfies
  \begin{equation}
  \alpha :=\tilde \gamma-\tilde r M e^{-R}>\alpha_{\textnormal{max}}+\lambda_N\zeta\beta\quad\text{and}\quad Z(0)=\Psi(0)<R, 
  \end{equation}
 where $M$ is an upper bound for the nonlinearity $f$. Over here we use the assumption that  $\Psi(0)$ is bounded from above by some constant $R'>0$.
  
  We are  now ready to exploit Lemma \ref{prop:MG} in a pathwise manner. We have, for every $t \geq 0$, that there exists\footnote{Indeed, the random variable is given by  $
	a^t=\sup\{s\in[0,t]:Z(s)<R\}.$}  a random variable $a^{t}\in[0,t]$  such that $\mathbb P$-a.s.\ the inequality
  \begin{equation}
  Z(t\wedge T_k) \leq \max \big\{R,R+\beta\zeta-\alpha (t-a^{t})+V(t\wedge T_k)-V (a^{t} \wedge T_k)\big\} 
  \end{equation}
holds,  for any integer $k\geq 1.$ Since $\alpha>\alpha_{\rm max},$ we particularly find
\begin{equation}
    \sup_{0\leq s\leq t} Z(t\wedge T_k)\leq R+\beta\zeta +\sup_{0\leq s<t\wedge T_\infty}\left|\int_0^tb(Z_{s-},s-)\,\mathrm dL(s)\right|,
\end{equation}
for any $k\geq 1.$ Along the lines of the proof of \cite[Prop. 3.4]{artikel1-general}, we conclude that finite time blowups towards $-\infty$ can only occur, hence there is no finite time blowup at $t=T_\infty$ (and  $Z(t)$ can therefore be extended to a global solution). 

The observations above imply that $\mathbb P$-a.s.\ we have
  \begin{equation}
  Z(t) \leq \max \left\{R,R+\zeta-\alpha \left(t-a^{t}\right)+V(t)-V (a^{t} )\right\}.\label{eq:neglect}
  \end{equation}
%
Introduce the process $U=(U(t))_{t\geq 0}$, defined by
  \begin{equation}
  U(t):=-\alpha  t+V(t), \quad t \geq 0.
  \end{equation}
  Then, for all $t \geq 0$, we obtain 
  \begin{equation}\label{theequation}
  	Z(t)  \leq \max \left\{R, R+\zeta+ U(t)-U (a^{t} )\right\} \\
  	\leq \max \left\{R, R+\zeta+\sup _{0 \leq s \leq t}(U(t)-U(s))\right\}.
  \end{equation}
  Observe that
  \begin{equation}
  U(t)=-\int_{0}^{t}\big(\alpha -a\left(Z_s,s\right)\big) \,\mathrm d s+\int_{0}^{t} b\left(Z_{s-},s-\right) \mathrm d L(s),\quad t\geq 0,
  \end{equation}
is a Lévy-driven process with negative drift. In particular, we have
  \begin{equation}
  \alpha -a\left(Z_{s},s\right) \geq \alpha -\alpha_{\textnormal{max}} >\lambda_N\zeta\beta\quad\text{and}\quad b\left(Z_{s-},s-\right)^{2} \leq \beta^{2},\quad s\geq 0.
  \end{equation}
  This allows us to appeal to Proposition \ref{prop:probabove} and we infer that
  $(\sup _{0 \leq s \leq t}(U(t)-U(s)))_{t \geq 0}$ is bounded above in probability. Together with  pathwise estimate \eqref{theequation}, we conclude|under the assumption that $\Psi(0)$ is bounded from above everywhere|that $(Z(t))_{t \geq 0}$ is bounded above in probability.

Now suppose the initial process $\Psi$ is arbitrary. It may be  possible that we  cannot bound $\Psi(0)$ from above uniformly on $\Omega$. In that case, define for every $R'>0$ the measurable set
  \begin{equation}
  	\Omega_{R'}:=\{\omega\in\Omega:\Psi(0,\omega)<R'\}.
  \end{equation}
The above implies that $(Z(t)\mathbf 1_{\Omega_{R'}})_{t \geq 0}$ is bounded in probability. Since $\mathbb P( \Omega_{R'}^c)\to 0$   as $R'\to\infty$, we  deduce that   $(Z(t))_{t \geq 0}$ is bounded above in probability, irrespective  of the initial data. 
\end{proof}



\begin{remark}	If the process $L=(L(t))_{t\geq 0}$  in Proposition \ref{thm:boundabove} is    a Brownian motion,   notice that the condition  in \eqref{eq:byproduct} simplifies to $\tilde \gamma >0$, which is conform    Proposition \ref{prop:alternatief}. 
\end{remark}
 
  
We are now ready to show the existence of an invariant measure  in  case  of $f(0)>0.$ In short, it suffices to show  that there is at least one initial condition $Y_0=\Phi\in\mathbb D[-1,0]$ for which the solution is bounded below in probability, as a result of Proposition \ref{thm:boundabove}. We will prove that all solutions are bounded below in probability, which  could also be achieved for non-autonomous systems. The tools  in {\S}\ref{Sec5} and {\S}\ref{sec:pathwise} are primarily developed to study  solutions close to $-\infty;$ see also \cite{BosGaaVer1-24}.




 \begin{proposition}
 \label{thm:boundbelow}
Consider the autonomous stochastic delay differential equation
  \begin{equation}
  \mathrm d Y(t)=\left[-\gamma+re^{-Y(t)}f\big(e^{Y(t-1)}\big)\right] \mathrm d t+a\left(Y_{t}\right) \mathrm d t+b\left(Y_{t-}\right) \mathrm d L(t),\label{eq:MGeq2}
  \end{equation}
where $\gamma,r>0.$ In addition to the assumptions in Proposition \ref{thm:boundabove}, let $f(0)>0$.
If there also exists a non-negative constant $ \alpha_{\rm min}\geq 0$ such that
\begin{equation}
 a(\varphi,\,\cdot\,)\geq  -\alpha_{\rm min}, \quad \text { for all } \varphi \in D[-1,0],
  \end{equation}
  then  all solutions persist globally and are  bounded in probability. Furthermore, there is an invariant measure, hence at least one stationary solution, to negative feedback system \eqref{eq:MGeq2}. 
\end{proposition} 

 \begin{proof} As before, we may assume without loss that $\Psi\in\mathbb D[-1,0]$ is bounded from below  uniformly on $\Omega$. Let $Z=(Z(t))_{-1\leq t<\infty}$ be  the solution to  \eqref{eq:MGeq2} with $Z_0=\Psi$, which
persists globally  and is bounded above in probability thanks to Proposition \ref{thm:boundabove}.
 
 Consider any real number $A>\alpha_{\textnormal{min}}+\lambda_N \zeta\beta $ and  introduce the process $V=(V(t))_{-1\leq t<\infty}$, defined for $t\geq 0$ by 
  \begin{equation}
  V(t)=
  	A t+\int_{0}^{t} a\left(Z_{s},s\right) \mathrm d s+\int_{0}^{t} b\left(Z_{s-},s-\right)\mathrm  d L(s), 
  \end{equation} 
  with $V(t)=0$ for $t\in[-1,0]$.
For every $t\geq t_0\geq 0$, we have
\begin{equation}
	Z(t )=Z(t_{0})+  \int_{t_{0}}^{t }\left[-(\gamma+A)+ r  e^{-Z(s)}f\big(e^{Z(s-1)}\big)\right]\mathrm d s \\
	+V(t)-V(t_{0}).
\end{equation}
Let us define the artificial mortality rate $\gamma_{A}=\gamma+A>0$.  
 Next,   take $R>0$    sufficiently large  such that that $\Psi(0)\geq - R$ holds.
Thanks to Corollary \ref{cor:q=0}, we obtain that for every $t \geq 0$ there exists a random variable  $a^{t}$ with values between 0 and $t$ such that
  \begin{equation}
  Z(t) \geq\min \big\{-R,-R-\gamma_A-\beta\zeta+V(t)-V (a^{t} )\big\}.
  \end{equation}
Now, for notational convenience, set $U(t):=-V(t),$ for $t\geq 0.$ This yields
  \begin{equation}
  Z(t) \geq\min\left\{-R,-R-\gamma_{A} -\beta \zeta-\sup _{0 \leq s \leq t}(U(t)-U(s))\right\}.
  \end{equation}
Observe that
  \begin{equation}
	U(t)=-\int_{0}^{t}\big(A+a\left(Z_s,s\right)\big)\, \mathrm d s+\int_{0}^{t}\big(- b\left(Z_{s-},s-\right)\big) \,\mathrm d L(s),\quad t\geq 0,
\end{equation}
is a Lévy-driven process with negative drift, since
\begin{equation}
	A+a\left(Z_{s},s\right) \geq A-\alpha_{\textnormal{min}} >\lambda_N\zeta\beta\quad\text{and}\quad \big(-b(Z_{s-},s-)\big)^{2} \leq \beta^{2},\quad s\geq 0.
\end{equation}
Proposition \ref{prop:probabove}, once again, tells us that
  $(\sup _{0 \leq s \leq t}(U(t)-U(s)))_{t \geq 0}$ is bounded above in probability.
This implies that $Z$ is bounded   in probability from below. An application  of  Corollary \ref{cor:application}   yields the existence of an invariant measure.
\end{proof}

\begin{proof}[Proof of Theorem \ref{main-thrm}] After the time transformation \eqref{eq:transformation} and setting $a=-\frac12b^2$, as in \eqref{eq:neg-drift}, we see that the result readily follows from combining Propositions \ref{thm:boundabove}  and  \ref{thm:boundbelow} for $f(0)>0$ and from Proposition  \ref{thm:boundabove} and Corollary \ref{cor:application} for $f(0)=0$.
\end{proof}


\printbibliography

\color{black}

\end{document}